\PassOptionsToPackage{table}{xcolor}
\documentclass{article}

\usepackage{iclr2027_conference,times}

\usepackage{amsmath}
\usepackage{amsfonts}
\usepackage{amsthm}
\usepackage{array}
\usepackage{graphicx}
\usepackage{mathtools}
\usepackage{booktabs}
\usepackage{longtable}
\usepackage{graphicx}
\usepackage{url}
\usepackage{microtype}
\usepackage{multirow}
\usepackage{nicefrac}
\usepackage{enumitem}
\usepackage{pifont}
\usepackage[normalem]{ulem}
\usepackage{xcolor}
\usepackage{url}            % simple URL typesetting
\DeclareRobustCommand{\cmark}{\ding{51}}
\DeclareRobustCommand{\xmark}{\ding{55}}

\newcommand{\rev}[1]{#1}
\newcommand{\del}[1]{}
\usepackage[hidelinks]{hyperref}
\usepackage[capitalize,noabbrev]{cleveref}

\newcommand{\orcid}[2]{\href{https://orcid.org/#1}{%
  \raisebox{-0.3ex}{\includegraphics[height=1.8ex]{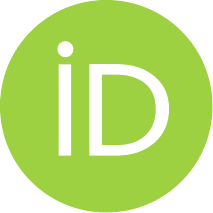}}\hspace{1mm}#2}}

\usepackage{tabularx}
\usepackage{threeparttable}

\newcolumntype{L}[1]{>{\raggedright\arraybackslash}p{#1}}
\newcolumntype{Y}{>{\raggedright\arraybackslash}X}

\theoremstyle{plain}
\newtheorem{theorem}{Theorem}[section]
\newtheorem{proposition}[theorem]{Proposition}
\newtheorem{lemma}[theorem]{Lemma}
\newtheorem{corollary}[theorem]{Corollary}
\theoremstyle{definition}
\newtheorem{definition}[theorem]{Definition}

\theoremstyle{remark}
\newtheorem{remark}[theorem]{Remark}

\iclrfinalcopy
\begin{document}

\title{Higher-Order Kolmogorov-Arnold Networks for Dynamics}

\author{
\orcid{0009-0008-3167-3992}{Kevin Slote} \\
Clarkson Center for Complex Systems Science \\
Clarkson University \\
Potsdam, NY, USA \\
\texttt{kslote@clarkson.edu}
\And
\orcid{0000-0002-1236-742X}{Jeremie Fish} \\
Clarkson Center for Complex Systems Science \\
Clarkson University \\
Potsdam, NY, USA \\
\texttt{jafish@clarkson.edu}
}

\maketitle

\begin{abstract}
The ability to derive the governing equations of a dynamical system from data is essential for prediction and control across many scientific fields. Two challenges, however, are the expansion of candidate terms when reducing higher-order equations to first-order form and noise from estimating derivatives. We propose Higher-Order-KANDy, which addresses both problems by using delay and derivative (``jet'') embeddings of the original data with derivative-ordered Kolmogorov-Arnold Networks for Dynamics (KANDy) layers stacked into a machine learning architecture. Compared with SINDy and WSINDy on noisy data, Higher-Order-KANDy learns governing equations without a candidate library (e.g., recovering the full-period $-\sin(u)$ with nothing named in advance), natively for higher-order systems, and holds the correct term structure as noise grows, selecting the exact active set on every seed at noise levels up to 7.5\% where an oracle weak-form regression admits a spurious term—though the oracle remains the more accurate coefficient estimator. We leverage the relationship between delay and differential embeddings, along with the Kolmogorov–Arnold representation theorem, to construct a delay-to-jet map and learn it as a Kolmogorov–Arnold network to preserve derivative structure and provide meaningful terms that appear in the governing equation. In our synthetic benchmarks, this construction recovers equations in the presence of coarse sampling and noise.
\end{abstract}

\section{Introduction}

Data-driven discovery of governing equations has reshaped forecasting and control grounded in physical law \citep{sindy,rudy2017data,lai2021finding,schmidt2009distilling,bongard2007automated,cranmer2023interpretable}. Sparse identification of nonlinear dynamics (SINDy) recovers interpretable equations by selecting a parsimonious set of terms from a user-specified candidate library~\citep{sindy}, and has been extended to PDEs~\citep{rudy2017data, schaeffer2017learning, li2024argosral} and to discovery from partial measurements~\citep{Bakarji2023DelayAutoencoders, brunton2017chaos, champion2019data}. Two limitations persist: noise amplified when estimating derivatives, especially second- and third-order ones, and the dictionary growth incurred when higher-order ODEs and PDEs are rewritten as first-order systems for sparse regression. Weak Form SINDy (WSINDy) addresses the first by replacing pointwise differentiation with linear transformations and variance reduction~\citep{messenger2021weaksindy}, with extensions to PDEs, interacting particle systems, and streaming data \citep{messenger2021weaksindypde,messenger2022meanfield,messenger2022online}. The remaining bottleneck is the upfront library choice: too small a library cannot express the equations, while too large a one degrades conditioning and interpretability and presupposes knowledge of the governing structure---a difficulty compounded combinatorially by the extra variables introduced in first-order reduction.

\begin{figure}
    \centering
    \includegraphics[width=1.0\linewidth]{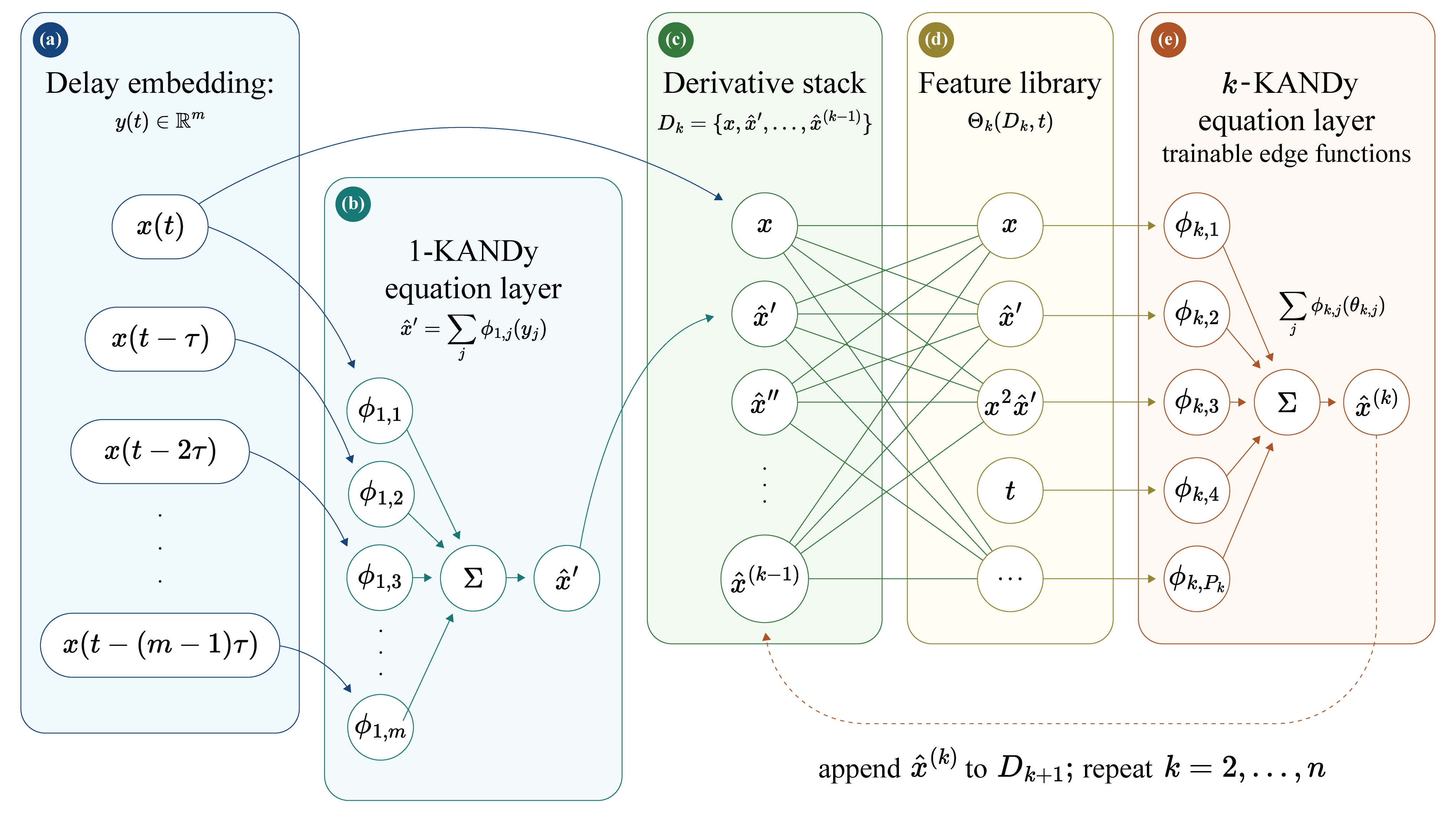}
    \caption{Overview of the Higher-Order-KANDy framework. \textbf{(a)} A delay embedding
$y(t) = (x(t), x(t-\tau), \dots, x(t-(m-1)\tau)) \in\mathbb{R}^m$ ($m$ delays,
step $\tau$) is passed through \textbf{(b)} a 1-KANDy lifting layer that estimates the
latent velocity $\hat{x}'$. \textbf{(c)} Repeated differentiation builds the derivative stack
$\mathcal{D}_k=\{x,\hat{x}',\ldots,\hat{x}^{(k-1)}\}$, from which \textbf{(d)} the feature set
$\Theta_k(\mathcal{D}_k,t)$ is assembled. \textbf{(e)} A $k$-KANDy equation layer learns
univariate edge functions $\varphi_{k,j}$ whose sum predicts $\hat{x}^{(k)}$, so the edges
read out directly as the additive terms of the governing equation.}
    \label{fig:higher-order-kandy}
\end{figure}

Separately, Kolmogorov--Arnold Networks (KANs) \citep{iclrkan2024,liu2024kan2}, inspired by the Kolmogorov--Arnold representation theorem \citep{Kolmogorov1957,Arnold2009_Kolmogorov}, place learnable univariate functions on network edges instead of fixed scalar activations, so each edge is a one-dimensional curve that can be inspected, plotted, and symbolically fit. Their theoretical grounding remains debated \citep{girosi1989representation,lin1993realization,koppen2002training,sprecher2002space,montanelli2020error,he2023optimal,lai2021kolmogorov}, yet KANs perform strongly across scientific domains \citep{scidirectkan2025,oxfordgenomicskan2025,naturegraphkan2025,coxkan2025}, with a growing set of architectural refinements \citep{Bagrow2025,Koenig2025,xkan2024}; in dynamical systems, KAN-based ODEs learn dynamics and hidden physics \citep{Koenig2024}.

The Kolmogorov--Arnold Network for Dynamics (KANDy) \citep{kandy} collapses the KAN to a single layer whose edges become the additive terms of the discovered equation \citep{kandy, deepkoopmankandy}. Like most equation-discovery methods, however, KANDy targets first-order ODEs, predicting a single time derivative from the current state; many physical systems are better described at higher order, and converting them to first-order form expands the dictionary and the room for error.

\paragraph{Related work.}
Scientific machine learning blends mechanistic structure with data-driven
modeling---neural ODE surrogates, physics-informed and operator-learning methods,
structure-preserving models, and Koopman
embeddings~\citep{karniadakis2021physics,chen2018neural,norcliffe2020second,raissi2019physics,lu2021deeponet,greydanus2019hamiltonian,williams2015data}---but
generally learns black-box vector fields. Symbolic regression and sparse equation
discovery instead seek explicit laws by selecting or searching over candidate
expressions~\citep{sindy,crutchfield1987equations,schmidt2009distilling,udrescu2020ai,cranmer2023interpretable}.
Complementary methods target derivative noise, higher-order equations, or partial observations through weak-form regression, differentiable symbolic networks, state inference, or surrogate differentiation~\citep{kacprzyk2023dcipher,martius2016extrapolation,sahoo2018learning,course2023state,kaheman2022automatic,hirsh2022sparsifying,messenger2021weaksindy,berg2019datadriven}.
Within the KAN line, \citet{Panahi2025KANModelDiscovery} fit KANs as
interpretable surrogates for non-sparse dynamics, without recovering closed-form
equations; KANDy~\citep{kandy} then collapsed the network to a
single layer whose univariate edges are the additive terms of the discovered
equation; we extend the construction here to higher order. SINDy-KANs~\citep{howard2026sindykan} (sparse regression per activation) and Symbolic--KAN~\citep{faroughi2026symbolickan} (discrete symbolic structure on the edges) pair KANs with discovery at first order; Note~\ref{sec:supp-backends} runs both on the sine-Gordon benchmark.
Higher-Order-KANDy differs in being end-to-end differentiable, with univariate
edges that directly represent additive equation terms while stacking one layer per derivative order.

\paragraph{Our contribution.}
We introduce \emph{Higher-Order-KANDy}, an architecture defined natively for higher-order ODEs and PDEs. We treat KANDy as a single layer of a deeper network whose $n$-th layer corresponds to the $n$-th derivative $x^{(n)}$ ($n$-KANDy), taking as input either a delay embedding \citep{takens1981, sauer1991embedology} or a \emph{jet embedding}---the observable's derivatives up to the needed order, read off a smooth fit to the data (Section~\ref{sec:delay})---so unobserved velocity and higher derivatives are recovered from scalar data. 
The result is an interpretable chain of single-derivative KAN modules whose edges read out as the additive terms of an $n$th-order governing equation. 
We also characterize the map from a delay window to the derivatives it encodes: it is the transition map between the delay and differential embeddings of the attractor.
%, whose departure from additive separability is proportional to the product of the nonlinearity and the sampling interval (Note~\ref{sec:supp-undersampled}, Lemma~\ref{lem:mixed}), so an additive (one-layer) KAN suffices for the former and a compositional (two-layer) KAN is needed for the latter; on the Lorenz observable at $\Delta t=0.25$ the two-layer map reduces $1-R^2$ on $x''$ from $0.099$ (additive) to $0.0037$. 
We demonstrate the architecture on eleven systems and probes spanning second- and third-order ODEs, nonlinear PDEs, and a hysteretic oscillator with a hidden non-derivative state; Supplementary Table~\ref{tab:systems} lists them with what each establishes.

\section{Methods}
\label{sec:methods}

We seek to discover an interpretable ordinary differential equation (ODE) directly from a scalar time series. Given samples $x(t_0),\dots,x(t_{T-1})$ of a single observable, sampled at uniform spacing $\Delta t$, we posit that the observable obeys an autonomous $n$-th order equation
\begin{equation}
  x^{(n)}(t) \;=\; F\!\big(x,\,x',\,\dots,\,x^{(n-1)}\big),
  \label{eq:nth-order}
\end{equation}
and our goal is to recover the right-hand side $F$ in a form whose individual additive terms can be read off and interpreted physically. 

The central object is Higher-Order-KANDy (Kolmogorov--Arnold Networks for Dynamics), a chain of single-derivative modules, in which each KANDy layer ($n$-KANDy) learns one step of the derivative hierarchy, and whose learnable univariate edges map directly onto the additive terms in \eqref{eq:nth-order}.

\subsection{Higher-Order-KANDy architecture}
\label{sec:arch}

Higher-Order-KANDy chains $n$ KANDy layers (Figure~\ref{fig:higher-order-kandy}), one per derivative order, so that the $k$-th layer predicts $x^{(k)}$ from the observable and the previously predicted lower derivatives.

\paragraph{Layer 1: lifting.}
\label{sec:lift}
The 1-KANDy layer is a black-box lifting map from the delay vector to the velocity,
\begin{equation}
  \hat{x}'(t) \;=\; \mathcal{L}\big(\mathbf{y}(t)\big),
  \label{eq:lift}
\end{equation}
implemented as a KAN layer with $m$ inputs and one output. By Section~\ref{sec:delay}, $x'$ is a function of the delay coordinates, so $\mathcal{L}$ recovers the unobserved velocity (in effect a learned, data-adapted finite-difference operator). For fully observed systems, where a directly measured velocity is available, this layer can be bypassed and the measured $x'$ supplied instead; this avoids contaminating a near-marginal equation with the lift's approximation error.

\paragraph{Layers $2,\dots,n$: equation layers.}
\label{sec:eqlayer}
Each higher module consumes a feature library of candidate nonlinear terms built from the running derivative stack and predicts the next derivative. For equation order $k$, let $\mathcal{D}_k = \{x,\,\hat{x}',\,\dots,\,\hat{x}^{(k-1)}\}$ be the available derivatives. A library $\Theta_k = \{\theta_{k,1},\dots,\theta_{k,P_k}\}$ is a list of named monomials up to order $k$ in these quantities; each term is a product of factors $x^{(r)}$ (the $r$-th derivative) raised to integer powers (e.g., $x$, $x'$, $x^2$, $x\,x'$, $x^2x'$), and may include nonlinear and transcendental functions.
%The equation layer (n-KANDy) is a KAN with one input edge per library term and a single output, so that the predicted derivative is an additive sum of univariate edges,
The equation layer (n-KANDy) learns the approximate symbolic form of the learned derivatives flowing through the network,
%is a KAN with one input edge per library term and a single output, so that the predicted derivative is an additive sum of univariate edges,
\begin{equation}
  \hat{x}^{(k)}(t)
  \;=\;
  \sum_{j=1}^{P_k} \varphi_{k,j}\!\big(\theta_{k,j}\big),
  \label{eq:eq-layer}
\end{equation}
each edge $\varphi_{k,j}$ being a learnable elementary function. In a canonical second-order case ($n=2$) one may take the Duffing--van der Pol library
$\Theta_2 = \{x,\,x',\,x^2x'\}$, giving the discovered equation
\begin{equation}
  \hat{x}'' \;=\;
  \varphi_{1}(x) \;+\; \varphi_{2}(x') \;+\; \varphi_{3}(x^2 x'),
  \label{eq:second-order}
\end{equation}
in which $\varphi_1$ is the restoring force, $\varphi_2$ the (linear) damping, and $\varphi_3$ the cubic damping / nonlinear cross-coupling; the forced-Duffing experiment instead uses edges on $x$, $x'$ and $t$, with the forcing carried by the time edge.
When the library is left empty, the equation layer reduces to free-form edges on the raw derivative coordinates $\mathcal{D}_k$, and a symbolic term such as $x^2$ is read off the shape of the corresponding edge rather than supplied as a column; this free-form additive discovery mode is what the canonical ODE and sine-Gordon experiments use.

Equation~\eqref{eq:eq-layer} is a sum of univariate functions, so a single-equation layer represents exactly those right-hand sides that are additively separable in their input coordinates. 
This is the condition under which the layer is free-form additively discoverable: when the governing law is a sum of terms each depending on one coordinate---as in $u_{tt}=u_{xx}-\sin(u)$, where the wave and restoring terms separate---each term is one edge, and no dictionary is needed. A term that couples coordinates multiplicatively, such as the transport $u\,u_x$ of Korteweg--de Vries or the cross terms $|\dot{x}|z^3$ of the Bouc--Wen memory law, is not of this form: it must be supplied as a composite feature $\theta_j$, which is what a candidate library does. 
%The library-free claims in this paper are therefore claims about additively separable dynamics; the experiments that supply a library are exactly those whose governing law is not of this form.

\subsection{Delay and differential embeddings: observability and noise}
\label{sec:delay}
%Two coordinate systems recur throughout the paper, so we fix the words once. 
The $p$-jet of the observable at time $t$ is simply the list of its derivatives up to order $p$, $(x, \dot{x}, \dots, x^{(p)})$ (for a field, all partial derivatives up to order $p$); it is what \eqref{eq:nth-order} needs on its right-hand side and what the equation layer consumes. 
The map that sends a state of the system to this list is the differential embedding; its counterpart built from time-shifted copies of $x$ is the delay embedding \citep{packard1980, takens1981, gibson1992}. 
The two front ends of Higher-Order-KANDy differ only in which of these coordinates they build from data.

A scalar measurement does not expose the higher derivatives required by \eqref{eq:nth-order}; the velocity $x'$ in particular is unobserved. We resolve this by Takens' delay embedding~\citep{takens1981}, forming the $m$-dimensional delay vector
\begin{equation}
  \mathbf{y}(t) \;=\;
  \big[\,x(t),\, x(t-\tau),\, x(t-2\tau),\,\dots,\, x(t-(m-1)\tau)\,\big]
  \in \mathbb{R}^{m},
  \label{eq:delay}
\end{equation}
with embedding dimension $m$ and lag $\tau$ in sample units. For a generic observable of a smooth dynamical system, the delay map embeds the underlying attractor, so $\mathbf{y}(t)$ furnishes coordinates in which the time derivatives of $x$ are functions of the delay coordinates---this is what makes the lifting step (Section~\ref{sec:lift}) well-posed for scalar data.

Partial observation is not the only barrier. The lifting layer reconstructs the velocity but is trained against pointwise finite-difference targets, and the noise in those targets grows with derivative order. When that noise dominates, we stop estimating derivatives one at a time. Instead, we fit a single smooth surrogate to the data---an analytic spline for fields, a learned map for scalar series---and read the whole $p$-jet off the surrogate by differentiating it rather than the data. We call this front end the \emph{jet embedding}: a differential embedding estimated from a smooth fit, whose output is the same set of derivative coordinates the delay lift produces, obtained without pointwise differencing~\citep{Whalen2015Observability}.

\paragraph{Analytic jet for fields (PDEs).}
For a field $u(x,t)$, the estimating spline must be two-dimensional, but every KAN edge is a one-dimensional spline: a single additive layer represents only sums $\varphi_1(x)+\varphi_2(t)$, never a general surface~\citep{kandy}; a learned per-line operator is also ill-posed, since different spatial locations have different amplitudes and shapes. We therefore use a single analytic surrogate: one tensor-product smoothing spline $S(x,t)$ of degree $(k_x,k_t)=(5,5)$~\citep{dierckx1981surface, dierckx1993curve}, fit to the noisy field with smoothing strength $\lambda = N\,\hat{\sigma}^2$ ($N$ the number of samples, $\hat{\sigma}^2$ from second differences along $t$). Every jet entry is then an exact partial derivative of that one fit,
\begin{equation}
  u \;=\; S,\qquad
  u_{xx} \;=\; \partial_x^2 S,\qquad
  u_{tt} \;=\; \partial_t^2 S,
  \label{eq:fieldjet}
\end{equation}
so a single smoothing-plus-differentiation step supplies the complete jet, including spatial derivatives, with noise suppressed by the smoothing fit rather than by differencing neighboring samples. For a scalar field $u(x,t)$ the $2$-jet at a grid point is $\big(u,u_x,u_t,u_{xx},u_{xt},u_{tt}\big)$. Computing it from the spline rather than learning it hands the equation layer a set of scalar coordinates at each sample point ($u$, $u_{xx}$, $u_{tt}$), and a scalar argument is exactly what a one-dimensional KAN edge requires, so the 2-KANDy layer consumes them natively, resolving the dimensional mismatch. For the undamped sine-Gordon collision, this yields $u_{tt}=c^2\,u_{xx}+\varphi(u)$, with the restoring force $\varphi$ a single free-form edge fit by the linear read-out.

\begin{figure}[t]
    \centering
    \includegraphics[width=0.46\linewidth]{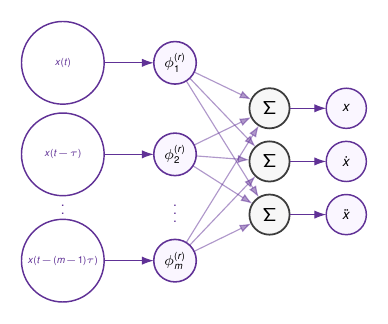}
    \includegraphics[width=0.46\linewidth]{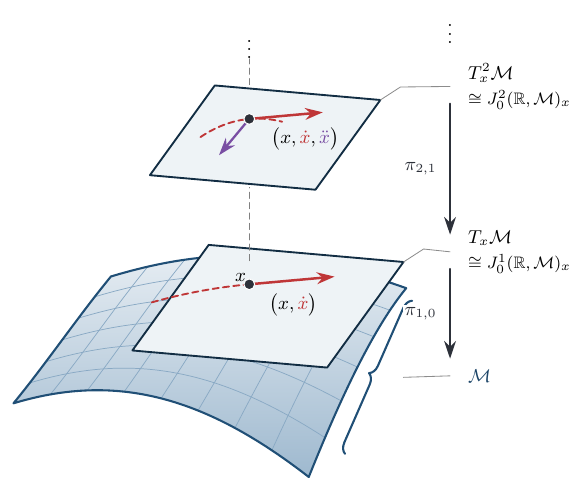}
    \caption{Learning mechanism for the KAN-based jet embedding. Left: the delay window $\mathbf{y}(t)$ flows through learned univariate KAN splines to the jet $(\hat{x},\hat{x}',\dots,\hat{x}^{(p)})$, so every derivative is read from one smooth map of the whole window rather than from pointwise differences. Right: the same map viewed as a change of chart on the attractor---from the delay embedding, which is what a scalar sensor measures, to the differential embedding, in which the governing equation takes the form of Eq.~\eqref{eq:nth-order}.}
    \label{fig:kan-jet}
\end{figure}

\paragraph{Kolmogorov--Arnold jet embeddings (scalar series).}
Two classical results make a learnable differential embedding principled (Fig.~\ref{fig:kan-jet}). For $m\geq 2D+1$ ($D$ the attractor dimension) the delay window determines the state, so the window-to-jet map exists and is continuous~\citep{takens1981, sauer1991embedology}, and the Kolmogorov--Arnold representation theorem~\citep{Kolmogorov1957} gives it the compositional (two-layer) form
\begin{equation}
  F_k(\mathbf{y}) \;=\; \sum_{q=1}^{2m+1} \Phi_{k,q}\!\left( \sum_{i=1}^{m} \phi_{q,i}(y_i) \right),
  \label{eq:kart}
\end{equation}
exactly a two-layer KAN. Writing $p$ for the highest order produced, the operator
\begin{equation}
  \mathcal{J}:\;\mathbf{y}(t)\;\longmapsto\;
  \big(\hat{x}^{(0)}(t),\,\hat{x}^{(1)}(t),\,\dots,\,\hat{x}^{(p)}(t)\big)
  \label{eq:jetop}
\end{equation}
is a KAN with $m$ input edges and $p+1$ outputs, the $k$-th estimating $x^{(k)}$. With linear edges it reduces to a finite-difference formula (a fixed weighted sum of the window samples); with spline edges it is the data-adapted generalization, denoising in the manner of a Savitzky--Golay filter. At dense sampling one additive (one-layer) KAN suffices (Lemma~\ref{lem:stencil}); on undersampled series the window-to-jet map is the transition map between the delay and differential embeddings, which an additive layer measurably fails to represent, and the two-layer form of Eq.~\eqref{eq:kart} is needed (Note~\ref{sec:supp-undersampled}: propositions and two supervised representation probes, on the Lorenz observable and on a coarse noisy sine-Gordon field). 
%After a warm-up, the spline knots are redistributed once to the embedded data (Section~\ref{sec:training}).

In either realization, the jet serves as the lifting layer, and everything downstream is unchanged. The substitution is exact when the system is $n$th-order in the observable, as for Duffing, Genesio--Tesi, and sine-Gordon in $u$; a hidden coordinate that is not a derivative of the observable is recovered by the trained Takens lift but not by the jet (Section~\ref{sec:res-boucwen}).

\subsection{Training objective}
\label{sec:training}
The model is trained end-to-end on a weighted sum of the lifting and equation losses. With $N$ aligned samples in the fitting window,
\begin{equation}
  \mathcal{L}
  \;=\;
  w_{\mathrm{lift}}\,
  \underbrace{\frac{1}{N}\sum_{t}\big(\hat{x}'(t)-x'(t)\big)^2}_{\text{lift loss}}
  \;+\;
  \underbrace{\frac{1}{N}\sum_{t}
    \frac{\big(\hat{x}^{(n)}(t)-x^{(n)}(t)\big)^2}{(\varsigma^{\mathrm{out}}_n)^2}}_{\text{equation loss}},
  \label{eq:loss}
\end{equation}
where the targets $x'$ and $x^{(n)}$ come from smoothed finite differencing or a Savitzky--Golay filter, the equation loss is normalized by the target variance, and $w_{\mathrm{lift}}$ balances velocity recovery against equation fit. When the lift is bypassed, only the equation loss is optimized, and the measured velocity is used in \eqref{eq:eq-layer}. 

\subsection{Baselines}
\label{sec:baselines}

On the PDE benchmark we compare pointwise SINDy/PDE-FIND~\citep{sindy,rudy2017data}, weak-form WSINDy~\citep{messenger2021weaksindy,messenger2021weaksindypde} with a naive polynomial library and with an oracle library containing $\sin(u)$, and D-CIPHER~\citep{kacprzyk2023dcipher}, all on identical noisy data; every method is scored on the recovered restoring force $R(u)$ against $-\sin(u)$ over the observed range, a design that separates derivative-noise error from library misspecification (configurations and D-CIPHER's model class in Note~\ref{sec:supp-baselines}).

\section{Results}
\label{sec:results}

\subsection{Systems analyzed and calibration on canonical oscillators}
\label{sec:res-ode}

\begin{figure}[t]
    \centering
    \includegraphics[width=0.9\linewidth]{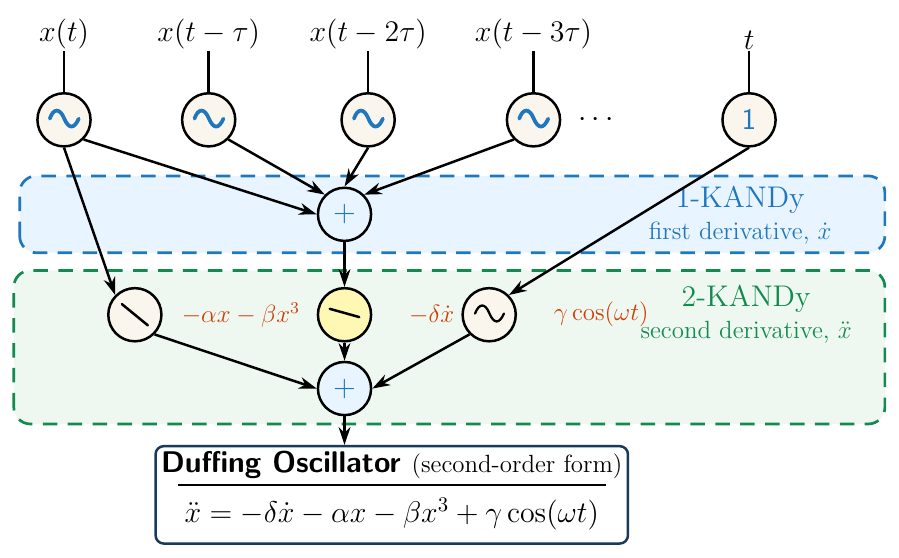}
    \caption{The Higher-Order-KANDy architecture applied to the Duffing oscillator. The inputs to the model are the states $x$ and the delay embedding $x(t - \tau)$ (or the jet-embedding alternative for high-noise data), as shown, and $t$. Each layer denoted $n$-KANDy predicts $x^{(n)}$. Each layer feeds the predicted derivative to the next layer.}
    \label{fig:duffing}
\end{figure}

Supplementary Table~\ref{tab:systems} lists every system, what it is meant to establish, and where its details appear. Three textbook second-order oscillators calibrate the read-out: the forced Duffing oscillator $\ddot{x} = -\alpha x - \beta x^3 - \delta\dot{x} + \gamma\cos(\omega t)$, the undamped pendulum $\ddot{\theta} = -(g/L)\sin\theta$, and the Rayleigh oscillator $\ddot{x} = -x + \mu(\dot{x}-\dot{x}^3)$. 
On each, the lifting layer recovers the velocity and the equation layer the second derivative, and each Layer-2 edge plotted against its argument overlays the corresponding analytic term, with absent terms collapsing to zero. 
Figure~\ref{fig:duffing} shows the architecture as applied to Duffing; Supplementary Figs.~\ref{fig:duffing-training}--\ref{fig:rayleigh-edges} and Table~\ref{tab:ode-coeffs} give the training curves, edge read-outs, and recovered coefficients.

\subsection{A native third-order system}
\label{sec:res-higher}

The architecture is defined for arbitrary order $n$ (Section~\ref{sec:arch}); the oscillators above, like most equation-discovery methods, exercise only $n=2$. 
The Genesio--Tesi system tests the native-order claim directly. From its scalar $x(t)$, with $\dddot{x} = -a\,\ddot{x} - b\,\dot{x} - c\,x + x^2$ ($a=1.2$, $b=2.92$, $c=6.0$)---and with $\dot{x}$ and $\ddot{x}$ never supplied---a delay lift ($m=7$, $\tau=6$) recovers both intermediate derivatives ($R^2=1.000$ and $0.999$) and a third-order equation layer predicts $\dddot{x}$ at $R^2=1.000$, reading the quadratic nonlinearity off the shape of the $x$-edge rather than from a supplied library column (Supplementary Fig.~\ref{fig:genesio-edges}); the recovered coefficients sit within roughly $8\%$ of truth, and free-running the discovered law reproduces the chaotic attractor. 
The same system coarsely sampled separates the front end from the equation representation (Note~\ref{sec:supp-controlled}): at two samples per swing an oracle-tuned Savitzky--Golay estimate of the derivatives loses three of the four terms, while a delay lift calibrated on a densely sampled trajectory recovers all four to within $1\%$ on a different, noise-free trajectory; oracle least squares and generic STLSQ on the lifted derivatives return the same equation, while the free-form additive discovery layer recovers the same terms with a seed spread of up to $14\%$ on the stiffness from equation-layer collinearity; the one-layer and two-layer lifts tie, as Lemma~\ref{lem:mixed} predicts for so weak a nonlinearity. Supporting evidence is in Supplementary Table~\ref{tab:genesio-coeffs} and
Figs.~\ref{fig:genesio-edges} and~\ref{fig:genesio}.

\subsection{Partial differential equations}
\label{sec:res-pde}

We apply the method to nonlinear PDEs by treating each interior spatial location as a scalar time series and supplying the spatial derivatives as additional equation-layer inputs. 
On the damped sine-Gordon equation $u_{tt} = c^2 u_{xx} - \sin(u) - \gamma u_t$ the equation layer recovers all three right-hand-side terms: the wave-speed term ($c^2=0.99$, true $1.0$) and the damping ($\gamma=0.20$, true $0.2$) accurately, and the periodic restoring nonlinearity $-\sin(u)$ (coefficient $-1.00$) across the full observed range $u\in[-1.8,\,11.6]$ (Supplementary Fig.~\ref{fig:sg-edges}).

\paragraph{Robustness under noise.}
On the Korteweg--de Vries equation $u_t = -6\,u\,u_x - u_{xxx}$, whose dispersive third derivative strongly amplifies measurement noise, neither pipeline escapes coefficient bias, but they fail differently: as noise rises from $1\%$ to $10\%$ the PDE-FIND (STLSQ) coefficient on $u\,u_x$ collapses from $-5.9$ to $-0.7$, spurious cross terms enter the support, and forward integration blows up by $t=0.08$, whereas the edge read-out degrades more slowly ($-4.2$ to $-2.6$), admits one spurious term, and forward-integrates a held-out two-soliton solution with the correct speed and shape (Supplementary Fig.~\ref{fig:kdv}).

\subsection{Noise robustness versus SINDy and WSINDy}
\label{sec:res-wsindy}

\begin{table}[t]
\centering\small
\begin{tabular}{llccccc}
\toprule
 & & Ground & SINDy & WSINDy & WSINDy & \textbf{Higher-Order-} \\
 & & truth & poly. & poly. & oracle & \textbf{KANDy} \\
\midrule
\multirow{4}{*}{\rotatebox[origin=c]{90}{\scriptsize{Derivatives}}} & $u_t$ & $\circ$ & \textendash & $\circ$ & $\circ$ & $\circ$ \\
 & $u_x$ & $\circ$ & \textendash & \textendash & \textendash & $\circ$ \\
 & $u_{xx}$ & \cmark & \cmark & \cmark & \cmark & \cmark \\
 & $u_{xt}$ & $\circ$ & \textendash & \textendash & \textendash & $\circ$ \\
\cmidrule(l){2-7}
\multirow{7}{*}{\rotatebox[origin=c]{90}{\scriptsize{Algebraic}}} & $1$ & $\circ$ & $\circ$ & $\circ$ & $\circ$ & \cellcolor{gray!12}$\circ$ \\
 & $u$ & $\circ$ & \xmark & \xmark & \xmark & \cellcolor{gray!12}$\circ$ \\
 & $u^{2}$ & $\circ$ & $\circ$ & $\circ$ & $\circ$ & \cellcolor{gray!12}$\circ$ \\
 & $u^{3}$ & $\circ$ & $\circ$ & \xmark & $\circ$ & \cellcolor{gray!12}$\circ$ \\
 & $u^{4}$ & $\circ$ & $\circ$ & $\circ$ & $\circ$ & \cellcolor{gray!12}$\circ$ \\
 & $u^{5}$ & $\circ$ & $\circ$ & $\circ$ & $\circ$ & \cellcolor{gray!12}$\circ$ \\
 & $\sin(u)$ & \cmark & $\oslash$ & $\oslash$ & \cmark & \cellcolor{gray!12}\cmark \\
\midrule
\multicolumn{2}{l}{TP / FP / FN} &  & 1 / 1 / 1 & 1 / 2 / 1 & 2 / 1 / 0 & \textbf{2 / 0 / 0} \\
\multicolumn{2}{l}{precision} &  & 0.50 & 0.33 & 0.67 & \textbf{1.00} \\
\multicolumn{2}{l}{$F_1$} &  & 0.50 & 0.40 & 0.80 & \textbf{1.00} \\
\bottomrule
\end{tabular}
\caption{Discovered term supports for the sine-Gordon kink--antikink collision, $u_{tt}=u_{xx}-\sin(u)$, at 5\% noise under the symmetric audit protocol: the Higher-Order-KANDy layer carries five free-form edges ($u$, $u_{xx}$, and the null controls $u_t$, $u_x$, $u_{xt}$) with no parsimony penalty on any edge, and a term counts for every method only if its coefficient exceeds $0.01$. \cmark: term selected; \xmark: spurious selected; \textcolor{gray}{$\oslash$}: true term never offered; $\circ$: offered, not selected; \textendash: outside the candidate space. Shaded cells: the branch a free-form additive discovery method searches with one learned $g(u)$. D-CIPHER appears as a reduced diagnostic in Supplementary Table~\ref{tab:sine-gordon-equations}. Precision separates all methods; recall separates only the polynomial libraries, whose missed $\sin(u)$ is a dictionary limitation. Restricted to its own candidate space each polynomial baseline's recall is $1.00$: they fail by coverage (grey $\oslash$), not selection.}
\label{tab:sg-term-support}
\end{table}

We quantify the restoring-force recovery against the sparse-regression baselines of Section~\ref{sec:baselines}, on the undamped sine-Gordon kink--antikink collision $u_{tt}=u_{xx}-\sin(u)$. This trajectory is chosen deliberately: at $x=0$ it sweeps $u$ over a $\pm2\pi$ period, the only regime in which $-\sin(u)$ is genuinely non-polynomial (a degree-five polynomial fits $-\sin$ to $\sim$$0.004$ RMS over $[-\pi,\pi]$ but degrades to $0.22$ over a full period), and the collision breaks the travelling-wave collinearity $u_{tt}=v^2u_{xx}$, so the wave-speed term $c\,u_{xx}$ and the restoring force $R(u)$ are separately identifiable. We fit $u_{tt}=c\,u_{xx}+R(u)$, the equation layer carrying five free-form edges---$u$, $u_{xx}$, and the null controls $u_t$, $u_x$, and $u_{xt}$---with no regularization on any edge (a sensitivity audit confirms the selection owes nothing to a sparsity prior; Note~\ref{sec:supp-equations}), and score two things separately: the support---whether a method selects exactly the true active set $\{u_{xx},\sin u\}$---and the form, the RMS error of the recovered $R(u)$ against $-\sin(u)$ over the observed range. The learned edges are splines, not symbolic expressions: each term's shape is recovered nonparametrically, and naming it---crediting the restoring edge as $\sin(u)$, or reading a slope---is a post-hoc identification against elementary functions, scored by the shape-correlation gate. A term counts for any method only if its coefficient exceeds $0.01$, one percent of the unit-scale true coefficients. Five methods are run at twelve noise levels from $0$ to $10\%$, with an independent noise realization drawn per (level, repeat) over five repeats and the network initialization varied with the repeat; D-CIPHER joins as a reduced diagnostic (one seed at four levels, Supplementary Table~\ref{tab:sine-gordon-equations}) (Table~\ref{tab:sg-term-support}; the full equation each method returned at every level is given in Supplementary Table~\ref{tab:sine-gordon-equations}).\footnote{\href{https://anonymous.4open.science/r/higher-order-kandy-iclr-733E/}{All code and experiments are available on GitHub.}}

The comparison separates two axes. (i) Free-form discovery versus library: the free-form KAN edge recovers the full-period $-\sin(u)$, whereas both polynomial-library baselines are stuck at one true term out of two at every noise level---a dictionary deficit the weak form cannot fix. (ii) Pointwise versus weak/analytic derivatives. Pointwise derivative estimates bottleneck the second-derivative target, and moving the derivatives off the data---onto a smooth surrogate (the analytic two-dimensional spline jet, Section~\ref{sec:delay})---removes this bottleneck.

On form the WSINDy oracle---handed the exact $\sin(u)$ column---is the better estimator at every noise level (Supplementary Table~\ref{tab:sine-gordon-equations}; Supplementary Fig.~\ref{fig:sg-collision}). At $5\%$ it returns $u_{tt}=-0.998\sin(u)+1.007\,u_{xx}$ (plus a spurious $-0.012\,u$), within $1\%$ of truth on the true terms, while the jet's coefficients attenuate to $-0.790$ and $0.879$, and by $10\%$ to $-0.695$ and $0.801$.

On support, the ordering reverses. Higher-Order-KANDy selects exactly $\{u_{xx},\sin u\}$ on all five seeds at every noise level through $7.5\%$, and on four of five at $10\%$. 
For every fit that holds the set, both criteria are met with wide margins: the restoring edge matches $-\sin(u)$ at $\rho \ge 0.952$ against an identification threshold of $0.90$, and all three null edges fall to activation $\le 0.12$ against a pruning threshold of $0.30$, while both live edges stay above $0.63$.
%on the fits that hold the set neither credit gate is close to a tie: the learned restoring edge matches $-\sin(u)$ at $\rho$ between $0.952$ ($1\%$) and ${\geq}0.989$ (all other levels) on the representative seeds against a $0.90$ threshold, and all three null edges prune at activation ${\leq}0.12$ against a $0.30$ threshold while both live edges stay above $0.63$. 
The oracle's support degrades instead: exact on all five seeds through $2\%$, it admits a small spurious linear $u$ (coefficient ${\approx}0.01$) on some seeds from $3\%$ upward, holding the exact set on two seeds in five at $5\%$ and $10\%$ (Supplementary Table~\ref{tab:sine-gordon-equations} reports agreement with the displayed seed's set, $3/5$, which at those levels includes the spurious $u$) (Table~\ref{tab:sg-term-support}; Supplementary Table~\ref{tab:sine-gordon-equations} gives the level-by-level detail). The polynomial libraries never recover the support at any noise level, because $\sin(u)$ is not in their span.

\paragraph{Backends on identical derivatives.}
Feeding every backend the same spline-jet samples the equation layer consumes isolates the representation (Note~\ref{sec:supp-backends}). Handed the $\sin(u)$ column, STLSQ on the same derivatives matches or beats the free-form edge on form at every noisy level while retaining a spurious linear $u$ at $5$--$10\%$, as the weak-form oracle does; without the column it fails by coverage even at zero noise. A re-implementation of SINDy-KAN~\citep{howard2026sindykan} and the authors' Symbolic--KAN code~\citep{faroughi2026symbolickan} on the same derivatives recover no exact support, for structural reasons the note gives.
%; the library-free layer adds nothing to coefficient accuracy once the right column is supplied, and what it adds is not needing the column.

\paragraph{Learning the derivative part.}
Replacing the damped sine-Gordon's linear dissipation with a saturating law $u_{tt}=u_{xx}-\sin(u)-\kappa\tanh(\nu u_t)$ puts a genuinely nonlinear function on the derivative axis; with the equation layer unchanged, all three terms are recovered from one noisy field ($\tanh$ damping to $\sim$$1\%$ relative error at zero noise, $3\%$ at $2\%$ noise), with nothing named in advance (Supplementary Note~\ref{sec:supp-tanh}).

\subsection{Hidden internal states beyond the first-order jet}
\label{sec:res-boucwen}

By the Takens embedding theorem, the delay lift reconstructs the full state from a scalar observable, including internal variables that are not derivatives of what we measure (Section~\ref{sec:delay}). 
The Bouc--Wen hysteretic oscillator tests this directly: its restoring force depends on a hidden memory variable $z$ with its own non-smooth evolution law, and traces a loop that is multivalued in $x$ and in the instantaneous pair $(x,\dot{x})$---the state a first-order-reduced model sees (Fig.~\ref{fig:boucwen}a). 
It is not multivalued in the full jet: by Eq.~\eqref{eq:bw-motion}, $z$ is an affine function of $(x,\dot{x},\ddot{x})$.
Reaching it that way, however, requires a clean second derivative and a third-order closure whose right-hand side ($|\dot{x}|\,|z|^{2}z$ with $z$ affine in the jet) is not additively separable, so a free-form additive discovery equation layer cannot represent it. 
The delay window supplies $z$ without either. 
The two front ends therefore separate sharply: the same KANDy layer is fit from each. 
Fed the exact instantaneous coordinates, it recovers the hidden state only partially ($R^2=0.88$) and the restoring force and acceleration barely at all ($R^2=0.23$ and $0.24$), whereas fed the delay window, it reconstructs $z$ almost exactly ($R^2=1.00$) and predicts both at $R^2=0.99$. 
Reading the edges off the reconstructed state then recovers both laws; with the true $z$ the motion law is within a percent of truth and the memory law within three, and in the observable-only pipeline the coupling holds to under a percent while the stiffness is $7\%$ and the damping $13\%$ high (Note~\ref{sec:supp-boucwen}). 
%\footnote{\href{https://anonymous.4open.science/r/higher-order-kandy-iclr-733E/}{All code and experiments are available on GitHub.}}

\section{Discussion}
\label{sec:discussion}
Across the systems of Supplementary Table~\ref{tab:systems} the same stacked construction recovers governing laws at their native order: textbook oscillator coefficients within a few percent (Duffing, pendulum) and within $11\%$ (Rayleigh), a genuinely third-order law from a scalar observable, PDE right-hand sides per spatial location, and---on sine-Gordon---the exact term support on every seed at noise levels up to $7.5\%$, learning $\sin(u)$ rather than being given it, where an oracle weak-form regression does not. 
%Because each module's univariate edges are the additive terms of the equation, every claim above is read off the model directly rather than inferred from predictions.

\paragraph{Parsimony and noise robustness.}
Sparse regression selects from a fixed library, and when high-order derivatives amplify noise, its coefficients drift, and forward integration diverges \citep{sindy,rudy2017data,lai2021finding}; Higher-Order-KANDy instead absorbs each right-hand-side contribution into one bounded edge and reads derivatives off a smooth surrogate, not the data. The two components pay on different axes: the edge recovers what no polynomial dictionary spans (the full-period $-\sin u$); 
the derivatives come from a smooth map rather than differenced data, and on coarsely sampled series only a learned delay-to-jet lift supplies them accurately (Note~\ref{sec:supp-undersampled}).
%the noise robustness is carried by differentiating the surrogate---a component any back end could adopt.
A practical corollary (Note~\ref{sec:supp-undersampled}): delay-based discovery has an optimal sampling rate and the learned map's tangent condition number is U-shaped in the sampling interval---and that number is computable from network and data alone.
%, a necessary though not sufficient indicator of sampling adequacy. 

\paragraph{Delay-to-jet embeddings.}
Proposition~\ref{prop:ka-delay-jet} establishes sufficient conditions under which the delay-to-jet map is a change of coordinates between delay and differential embeddings and admits a two-layer Kolmogorov--Arnold representation (Figure~\ref{fig:kan-jet}). 
Lemma~\ref{lem:mixed} supplies the complementary lower bound by showing that a single additive layer can be insufficient. 
Thus, two layers are sufficient in general and become necessary within the nonseparable regime identified by the lemma. 
\citet{packard1980} proposed two ways to reconstruct dynamics from a single signal—time delays and successive derivatives—and the field settled on delays because derivatives amplify noise~\citep{gibson1992, Krakavoska2022}, even though derivatives are the coordinates in which physical laws are written, carry physical units, and reduce the entire vector field to a single closure $x^{(n)}=F$.
Higher-Order-KANDy recovers those coordinates by learning the change of representation between the two embeddings: a delay-to-jet map that admits an exact two-layer Kolmogorov–Arnold representation and, on coarsely sampled data, supplies derivatives where fixed differentiation formulas fail.

\paragraph{Limitations.}
 First, where a term library $\Theta_k$ is supplied, the equation layer assumes it; though smaller than a sparse-regression dictionary, an inapt library still misattributes the dynamics. 
 Second, under partial observability, a hidden coordinate that is not a derivative of the observable---the Bouc--Wen memory variable $z$---requires the delay lift, and that lift is a supervised probe: it recovers $z$ at $R^2=0.998$, whereas an unsupervised latent matches $z$ in shape only (Note~\ref{sec:supp-boucwen}); the coefficient read-outs therefore condition on the reconstructed state. 

\bibliographystyle{iclr2027_conference}
\bibliography{references}

\section*{Use of Large Language Models}

We used a generative AI assistant to restructure this manuscript
for readability and length, to draft several passages, to generate code that we validated and ran ourselves, and as an adversarial fact-checker of the text, figures, and tables. We reviewed all AI-assisted work and take responsibility for the paper's final content. 
No scientific content was devised by an LLM.

%\clearpage
\setcounter{section}{0}
\renewcommand{\thesection}{S\arabic{section}}
\setcounter{figure}{0}
\renewcommand{\thefigure}{S\arabic{figure}}
\setcounter{table}{0}
\renewcommand{\thetable}{S\arabic{table}}

\begin{table}[t]
\centering\footnotesize
\setlength{\tabcolsep}{4pt}
\caption{Every system analyzed, the native order $n$ at which it is treated, what the model observes, and what the experiment establishes. Rows above the rule are calibration, native-order, and representation checks reported in full in the supplement; the boldface rows below it carry the paper's headline claims and are detailed in the main text. Hyperparameters for the trained models are listed in Supplementary Table~\ref{tab:supp-hyperparams}.}
\label{tab:systems}
\begin{tabular}{@{}l c l p{5.1cm} l@{}}
\toprule
System & $n$ & Obs. & What it establishes & Details \\
\midrule
Duffing (forced)      & 2 & $x$        & all four coefficients within a few \% & Tab.~\ref{tab:ode-coeffs}, Fig.~\ref{fig:duffing-edges} \\
Pendulum              & 2 & $\theta$   & full $-(g/L)\sin\theta$ edge; absent damping collapses to zero & Tab.~\ref{tab:ode-coeffs}, Fig.~\ref{fig:pendulum-edges} \\
Rayleigh              & 2 & $x$        & nonlinearity carried on the velocity edge, within $11\%$ & Tab.~\ref{tab:ode-coeffs}, Fig.~\ref{fig:rayleigh-edges} \\
Genesio--Tesi         & 3 & $x$        & native third order from a scalar; coefficients within $8\%$, attractor reproduced; survives two samples per swing & Tab.~\ref{tab:genesio-coeffs}, Fig.~\ref{fig:genesio}, Note~\ref{sec:supp-controlled} \\
Korteweg--de Vries    & 1 & $u(x,t)$   & degrades slowly where PDE-FIND diverges under noise & Fig.~\ref{fig:kdv} \\
Sine--Gordon (damped) & 2 & $u(x,t)$   & all three right-hand-side terms recovered & Fig.~\ref{fig:sg-edges} \\
Sine--Gordon (tanh damping) & 2 & $u(x,t)$ & nonlinear damping $-\kappa\tanh(\nu u_t)$ learned on the derivative axis & Note~\ref{sec:supp-tanh} \\
Lorenz (probe)        & --- & $x$      & delay-to-jet map: additive when dense, compositional when undersampled & Note~\ref{sec:supp-undersampled} \\
Sine--Gordon (coarse, noisy; probe) & 2 & $u(x,t)$ & coefficients recovered from $80\times33$ samples at $1\%$ noise by a compositional (two-layer) patch operator calibrated on exact solutions & Note~\ref{sec:supp-undersampled} \\
\midrule
Sine--Gordon (collision) & 2 & $u(x,t)$ & exact term support to $7.5\%$ noise, library-free & Section~\ref{sec:res-wsindy} \\
Bouc--Wen             & $2{+}1$ & $x$  & \textbf{hidden non-derivative state supplied by the delay lift} & Section~\ref{sec:res-boucwen}, Note~\ref{sec:supp-boucwen} \\
\bottomrule
\end{tabular}
\end{table}

\begin{table}[t]
  \centering
  \begin{tabular}{llrr}
    \toprule
    System & Term & Recovered & True \\
    \midrule
    Duffing  & $\alpha$ (restoring, $x$)       & $1.023$  & $1.0$ \\
             & $\beta$ (cubic, $x^3$)          & $0.962$  & $1.0$ \\
             & $\gamma$ (forcing)              & $0.492$  & $0.5$ \\
             & $\delta$ (damping)              & $0.294$  & $0.3$ \\
    \midrule
    Pendulum & $g/L$                           & $9.800$  & $9.810$ \\
             & damping                         & $0.0000$& $0.0$ \\
    \midrule
    Rayleigh & stiffness                       & $-0.900$ & $-1.0$ \\
             & $\mu_{\mathrm{lin}}$            & $0.910$  & $1.0$ \\
             & $\mu_{\mathrm{cub}}$            & $-0.890$ & $-1.0$ \\
    \bottomrule
  \end{tabular}
    \caption{Recovered versus true coefficients for the canonical ODEs.}
  \label{tab:ode-coeffs}
\end{table}

\begin{figure}[t]
  \centering
  \includegraphics[width=0.75\linewidth]{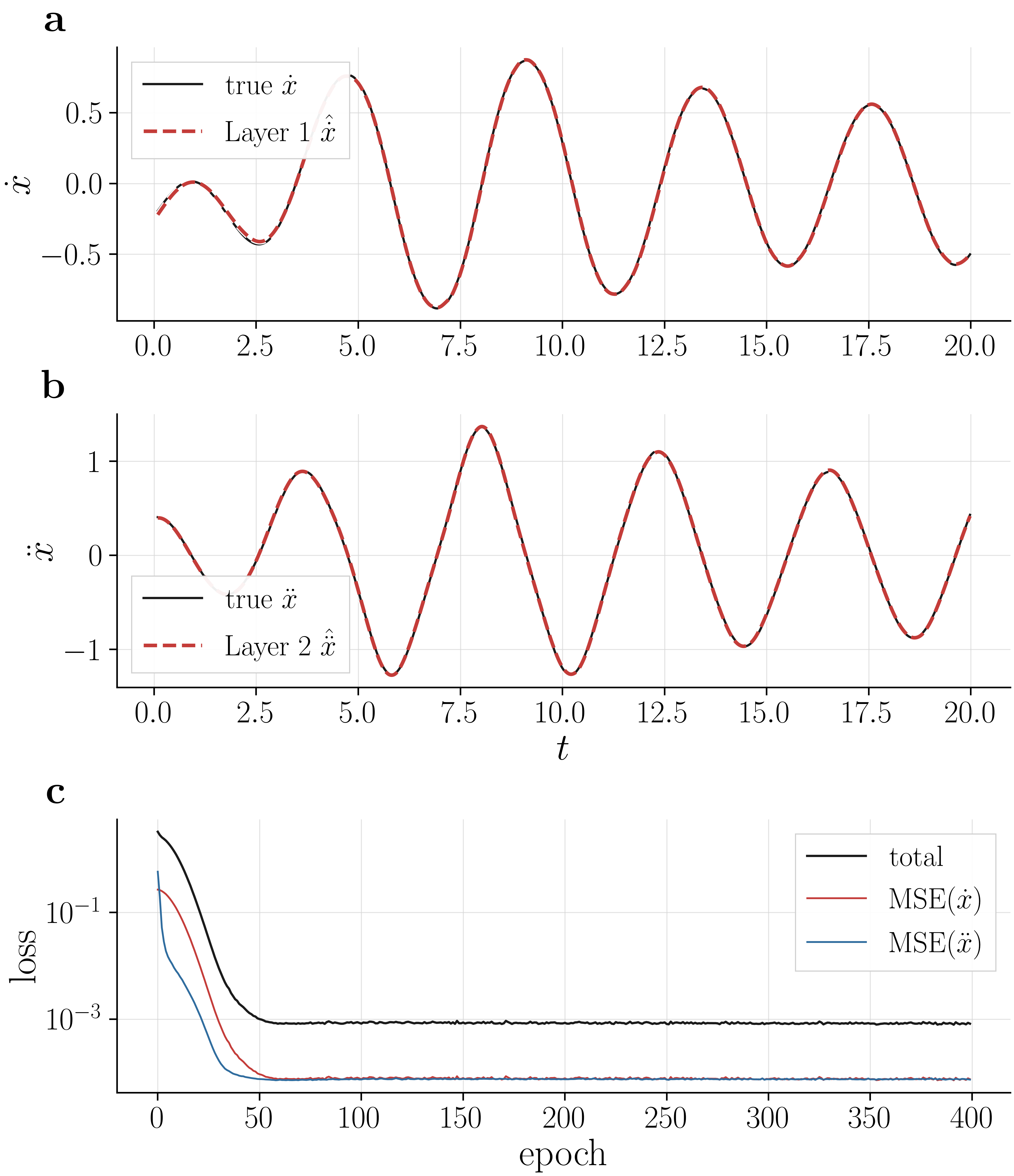}
  \caption{\textbf{Duffing oscillator: trajectory reconstruction.} Top: true
    velocity $\dot{x}$ versus the Layer-1 (lifting) prediction. Middle: true
    acceleration $\ddot{x}$ versus the Layer-2 (equation) prediction. Bottom:
    training losses (total, lift MSE, equation MSE) on a log scale. The two
    learned curves are visually indistinguishable from ground truth and the loss
    converges within $\sim$100 epochs.}
  \label{fig:duffing-training}
\end{figure}

\begin{table}[t]
  \centering
  \caption{Recovered versus true coefficients for the third-order
    Genesio--Tesi system, from the scalar observable $x(t)$ alone.}
  \label{tab:genesio-coeffs}
  \begin{tabular}{llrr}
    \toprule
    Term & Role & Recovered & True \\
    \midrule
    $a$ & $\ddot{x}$ damping     & $1.29$ & $1.20$ \\
    $b$ & $\dot{x}$ damping      & $2.90$ & $2.92$ \\
    $c$ & $x$ stiffness          & $6.39$ & $6.00$ \\
    $q$ & $x^2$ nonlinearity     & $1.06$ & $1.00$ \\
    \bottomrule
  \end{tabular}
\end{table}

\begin{figure}[t]
  \centering
  \includegraphics[width=\linewidth]{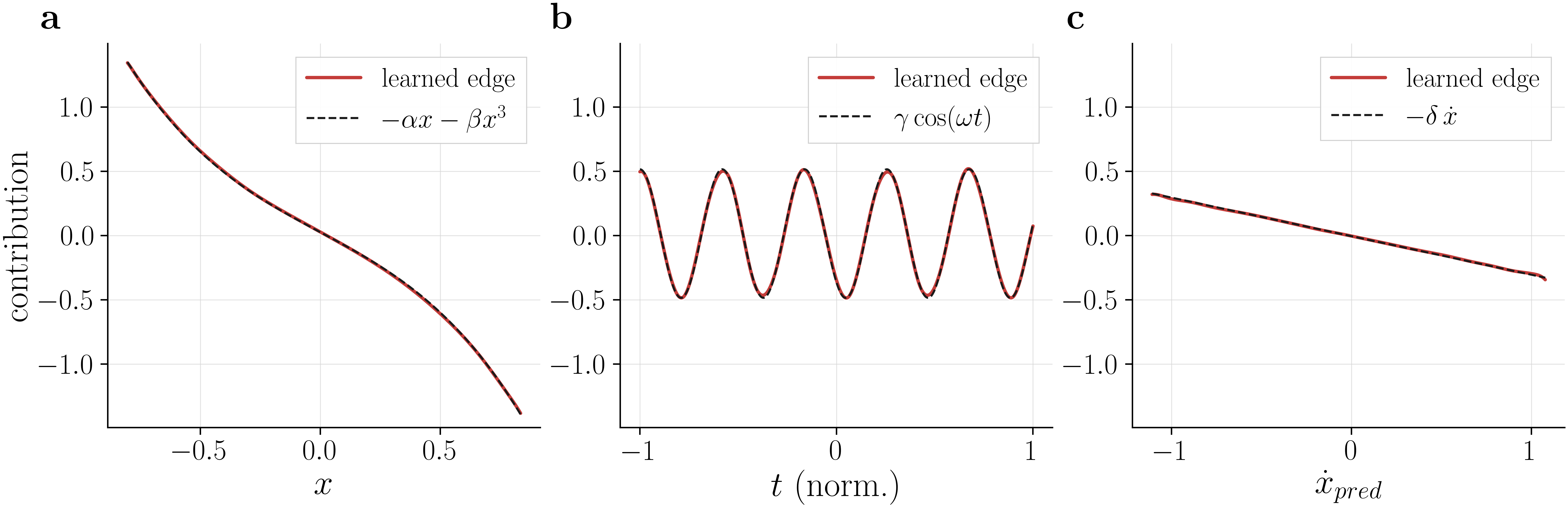}
  \caption{\textbf{Duffing: recovered equation terms.} Learned Layer-2 edges
    (red) versus analytic terms (black dashed), mean-removed. Recovered
    coefficients: $\alpha=1.023$, $\beta=0.962$ (true $1.0,1.0$);
    $\gamma=0.492$ (true $0.5$); $\delta=0.294$ (true $0.3$).}
  \label{fig:duffing-edges}
\end{figure}

\begin{figure}[t]
  \centering
  \includegraphics[width=0.85\linewidth]{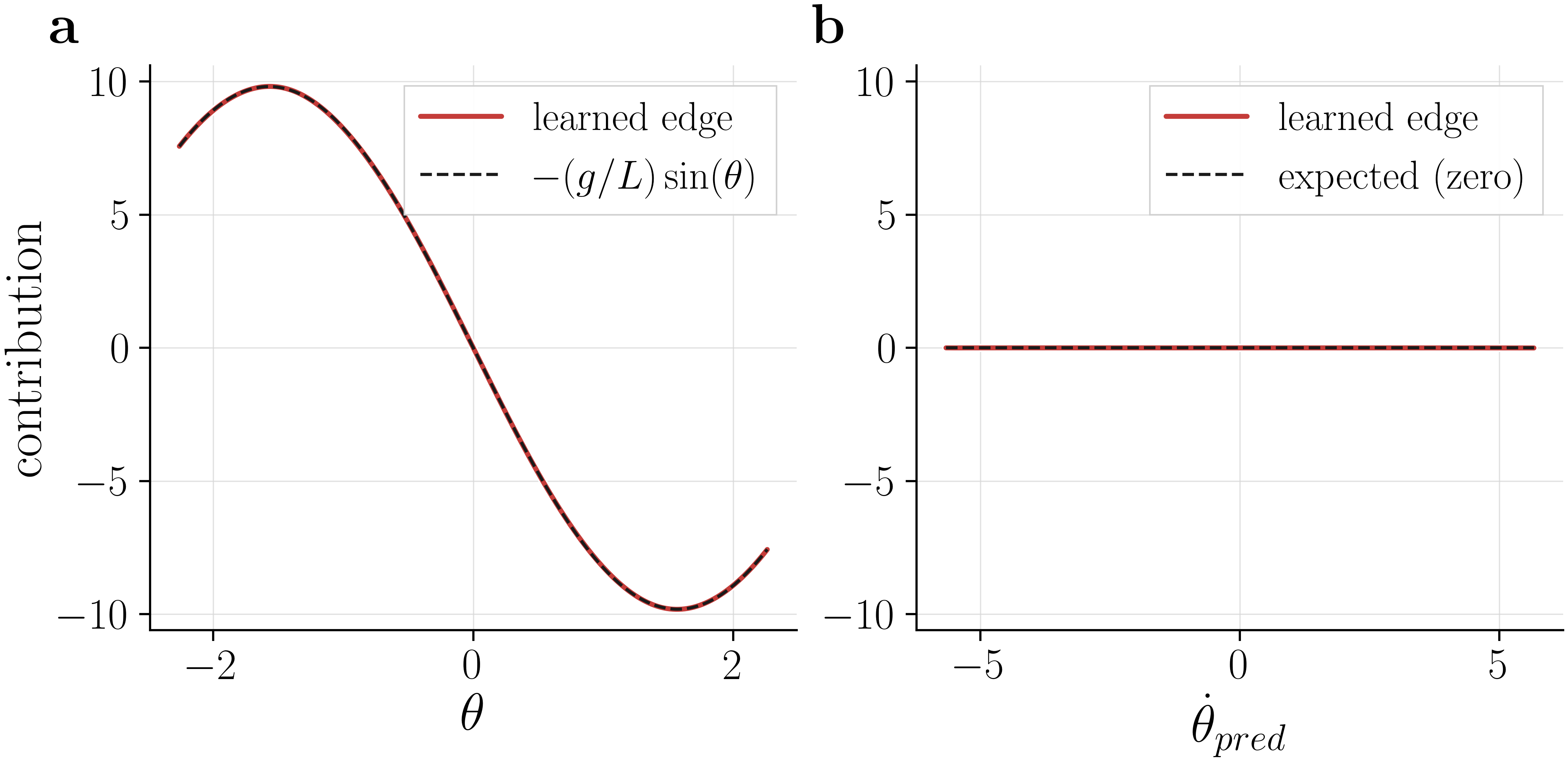}
  \caption{\textbf{Pendulum: recovered equation terms.} The state edge recovers $-(g/L)\sin\theta$ with $g/L=9.800$ (true $9.810$); the velocity edge is flat, giving a damping coefficient of $0.0$ (true $0$).}
  \label{fig:pendulum-edges}
\end{figure}

\begin{figure}[t]
  \centering
  \includegraphics[width=0.85\linewidth]{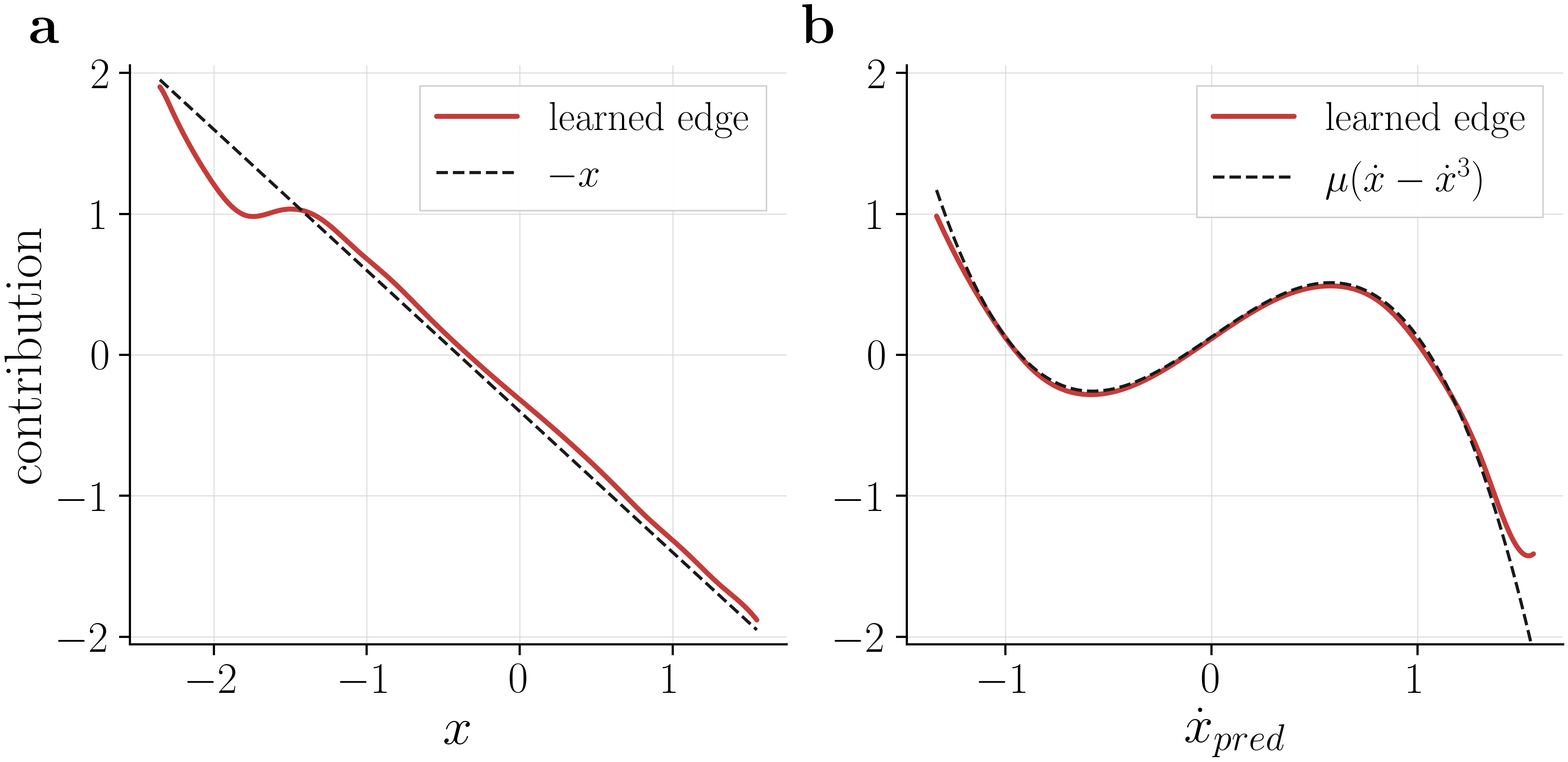}
  \caption{\textbf{Rayleigh oscillator: recovered equation terms.} Linear
    restoring force (stiffness $-0.900$, true $-1.0$) and cubic nonlinear damping
    on the velocity edge ($\mu_{\mathrm{lin}}=0.910$,
    $\mu_{\mathrm{cub}}=-0.890$; true $1.0,-1.0$).}
  \label{fig:rayleigh-edges}
\end{figure}

\begin{figure}[t]
  \centering
  \includegraphics[width=\linewidth]{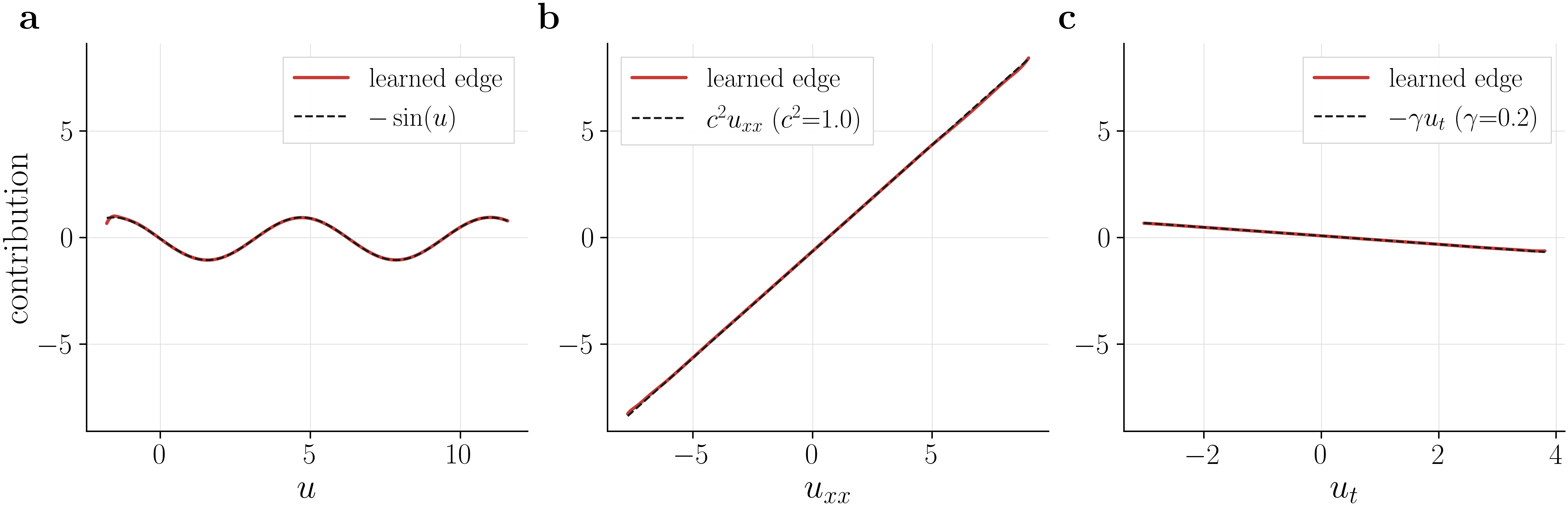}
  \caption{\textbf{Sine-Gordon PDE: recovered equation terms.} Learned edges
    (red) versus analytic terms (black dashed) for the $u$, $u_{xx}$, and $u_t$
    inputs. All three terms are
    recovered accurately ($c^2=0.99$, $\gamma=0.20$, $\sin$ coefficient
    $-1.00$), the restoring nonlinearity tracking $-\sin(u)$ over the full
    observed range $u\in[-1.8,\,11.6]$.}
  \label{fig:sg-edges}
\end{figure}

\begin{figure}[t]
  \centering
  \includegraphics[width=\linewidth]{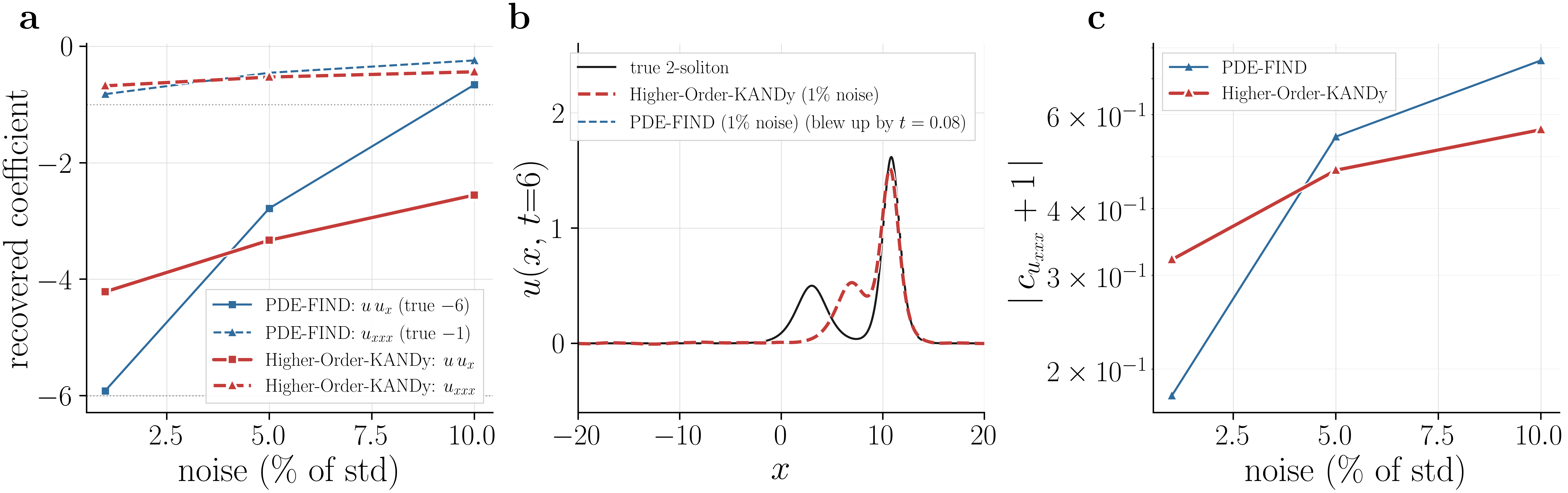}
  \caption{\textbf{KdV: robustness to noise in the third spatial derivative.}
    (A) Recovered coefficients versus noise: both methods drift from truth, the baseline faster and with spurious support entering at every level. (B) Forward integration of a held-out two-soliton solution: the baseline blows up by $t=0.08$, the proposed recovery matches speed and shape (RMS $0.165$ at $t=6$, $1\%$ noise). (C) $u_{xxx}$ coefficient error versus noise (log scale).}
  \label{fig:kdv}
\end{figure}

\begin{figure}[t]
  \centering
  \includegraphics[width=\linewidth]{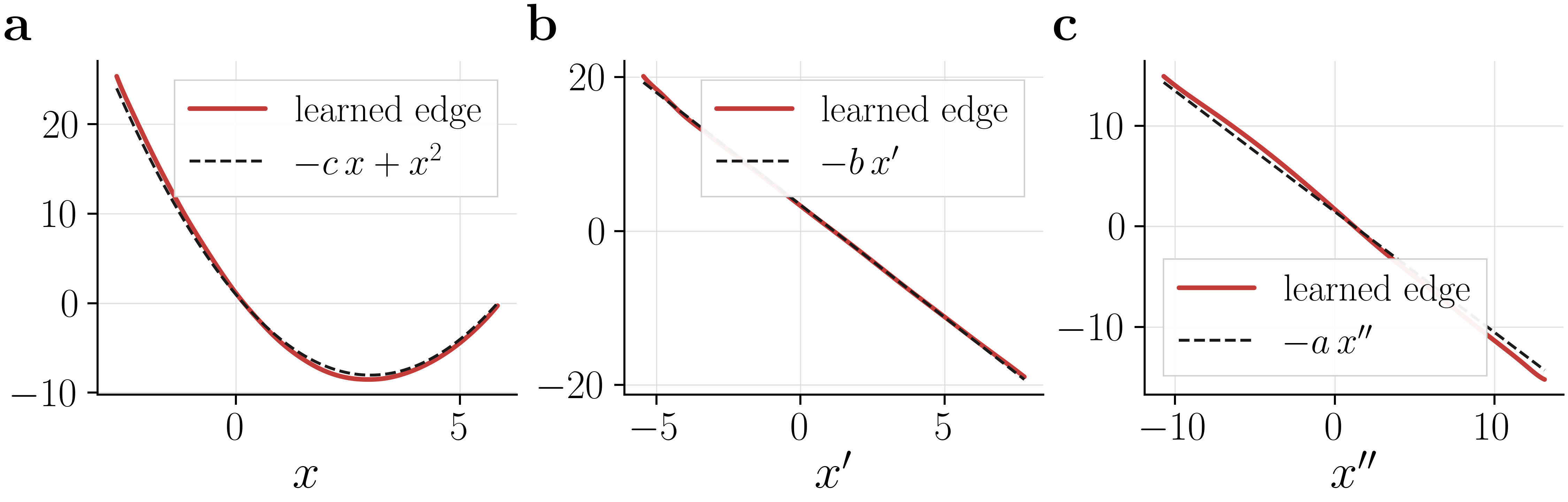}
  \caption{\textbf{Third-order discovery (Genesio--Tesi): recovered equation terms.}
    Learned equation-layer edges (red) against their analytic terms (black dashed).
    \textbf{a:} the $x$ edge carries the stiffness and the quadratic nonlinearity,
    $-c\,x + x^{2}$, read off the shape of the edge rather than supplied as a
    library column. \textbf{b:} the $x'$ edge carries the linear damping $-b\,x'$.
    \textbf{c:} the $x''$ edge carries $-a\,x''$.}
  \label{fig:genesio-edges}
\end{figure}

\begin{figure}[t]
  \centering
  \includegraphics[width=0.85\linewidth]{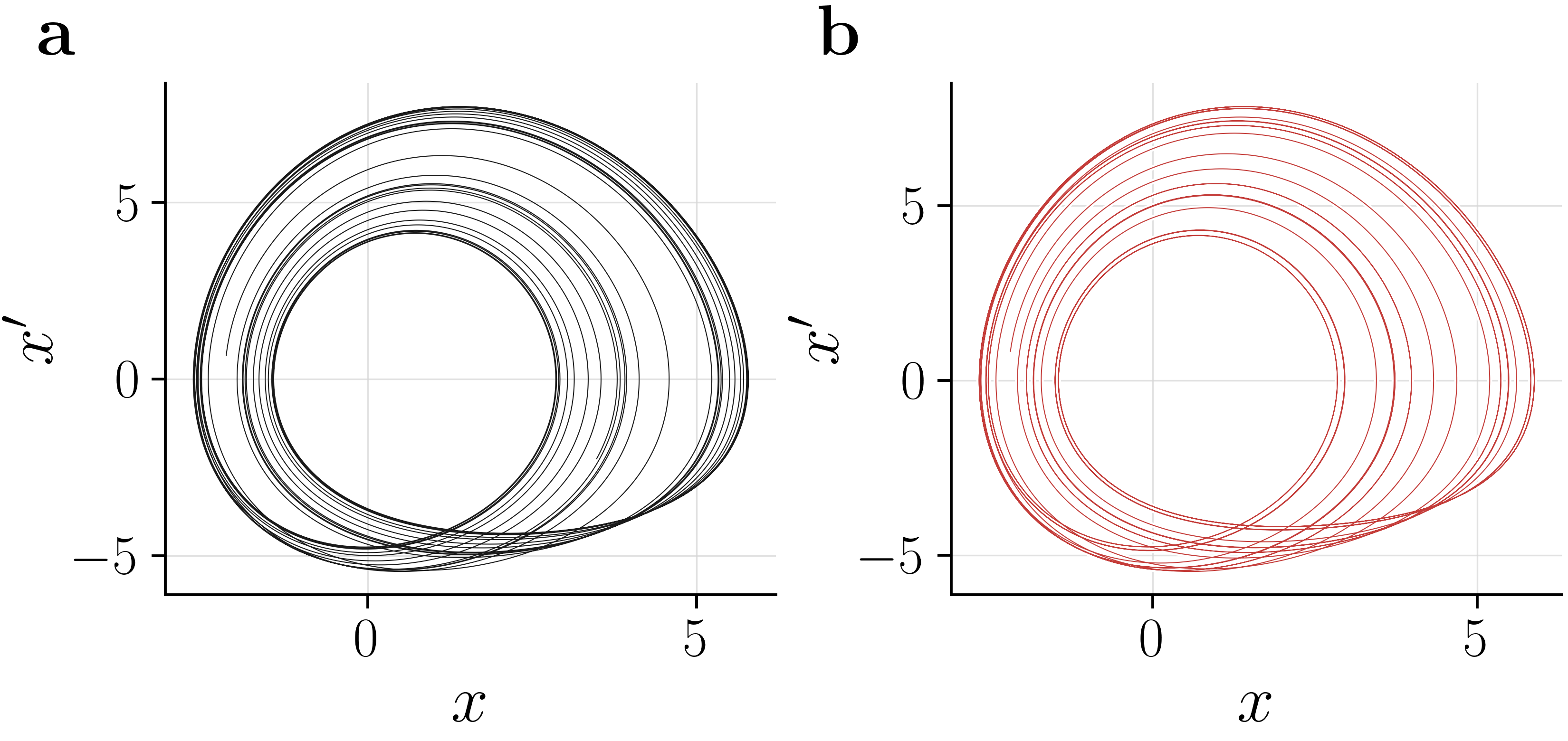}
  \caption{\textbf{Third-order discovery (Genesio--Tesi): free-run attractor.}
    \textbf{a:} the true Genesio--Tesi chaotic attractor in the $(x,x')$ plane.
    \textbf{b:} the attractor produced by freely integrating the discovered
    third-order equation from a single initial condition.}
  \label{fig:genesio}
\end{figure}

\begin{figure}[t]
  \centering
  \includegraphics[width=\linewidth]{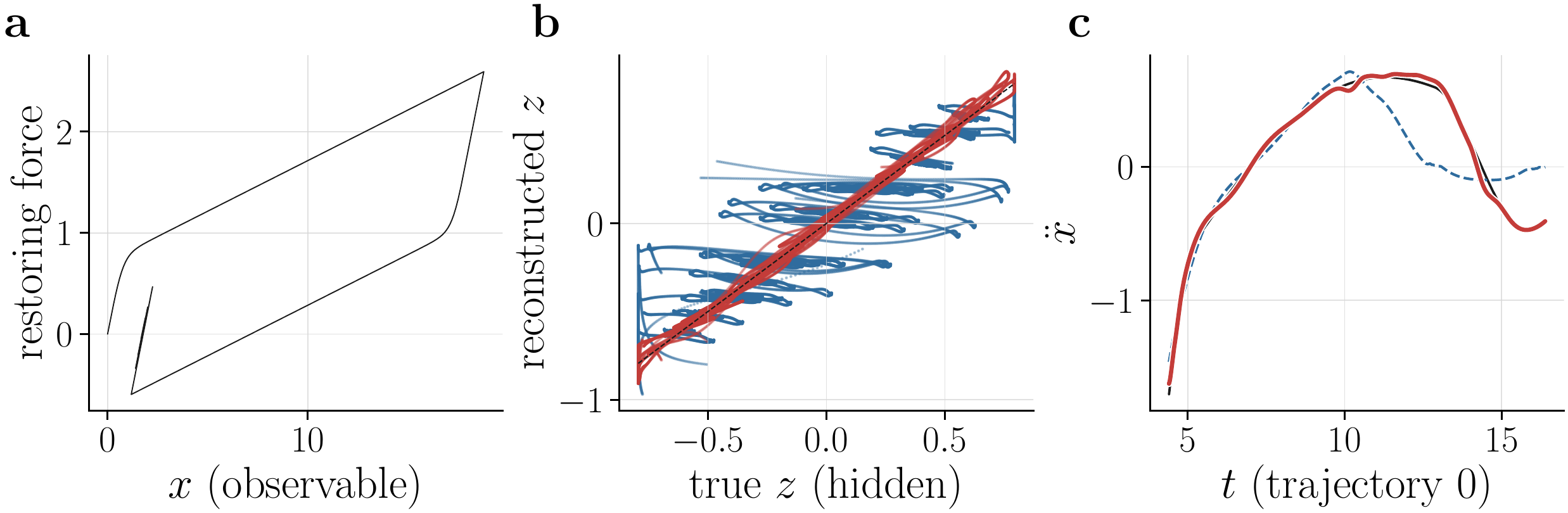}
  \caption{\textbf{Bouc--Wen hysteresis and hidden memory variables.}
    \textbf{a:} the restoring force traces a hysteresis loop---it is multivalued
    in the observable $x$ and in $(x,\dot{x})$, so no function of the
    instantaneous pair can represent it. \textbf{b:} reconstruction of the hidden memory variable $z$: from the
    instantaneous coordinates $(x,\dot{x})$ (blue) the fit is spread and biased
    at the load reversals, whereas from the delay window (red) it collapses onto
    the diagonal ($R^2=1.00$). \textbf{c:} the acceleration $\ddot{x}$ over one
    trajectory---the Takens delay lift (red) tracks the truth (black) while the
    instantaneous predictor (blue) fails.}
  \label{fig:boucwen}
\end{figure}

\begin{table}[t]
  \centering
  \begin{tabular}{llrr}
    \toprule
    Law & Term & Recovered & True \\
    \midrule
    Motion \eqref{eq:bw-motion} & $\alpha k$ (restoring, $x$)        & $0.100$ & $0.100$ \\
                                & $c$ (damping, $\dot{x}$)          & $0.151$ & $0.150$ \\
                                & $(1-\alpha)k$ (coupling, $z$)     & $0.900$ & $0.900$ \\
    \midrule
    Memory \eqref{eq:bw-memory} & $A$ ($\dot{x}$)                   & $0.99$  & $1.0$ \\
                                & $\beta$ ($|\dot{x}|z^3$)          & $0.97$  & $1.0$ \\
                                & $\gamma$ ($\dot{x}|z|^3$)         & $1.01$  & $1.0$ \\
    \bottomrule
  \end{tabular}
  \caption{Recovered versus true coefficients for the Bouc--Wen oscillator.
    The motion-law coefficients shown are read with the reconstructed state
    (true $z$); the observable-only pipeline (delays $\to\hat{z}\to$ equation
    layer) gives $-0.107$, $-0.170$, $-0.906$ (Note~\ref{sec:supp-boucwen}). The memory law is
    identified given the lifted state, which the delays recover at
    $R^2=0.998$.}
  \label{tab:boucwen-coeffs}
\end{table}

\begin{table}[t]
  \centering
  \small
  \setlength{\tabcolsep}{5pt}
  \begin{tabular}{llccccr r}
    \toprule
    Experiment & Basis & $n$ & Grid & $m/\tau$ & Epochs & LR & Batch \\
    \midrule
    \multicolumn{8}{@{}l}{\emph{Canonical second-order ODEs}}\\
    Duffing        & B-spline & 2 & 25 & $4/3$ & 400  & $2\times10^{-3}$ & 128 \\
    Pendulum       & B-spline & 2 & 25 & $3/5$ & 400  & $2\times10^{-3}$ & 128 \\
    Rayleigh       & B-spline & 2 & 25 & $4/3$ & 400  & $2\times10^{-3}$ & 128 \\
    \midrule
    \multicolumn{8}{@{}l}{\emph{Higher-order systems}}\\
    Genesio--Tesi  & B-spline & 3 & 25 & $7/6$ & 1200 & $2\times10^{-3}$ & 512 \\
    \midrule
    \multicolumn{8}{@{}l}{\emph{Nonlinear PDEs}}\\
    Sine--Gordon (damped)   & B-spline & 2 & 35 & $4/3$ & 400 & $2\times10^{-3}$ & 256 \\
    Sine--Gordon (tanh damping) & B-spline & 2 & 25 & -- & 400 & $2\times10^{-3}$ & 256 \\
    Sine--Gordon (WSINDy benchmark) & B-spline & 2 & 25 & -- & 400 & $2\times10^{-3}$ & 256 \\
    Korteweg--de Vries      & B-spline & 1 & 25 & -- & 400 & $10^{-3}$ & 256 \\
    \midrule
    \multicolumn{8}{@{}l}{\emph{Additional systems}}\\
    Bouc--Wen hysteresis & B-spline & 2 & 12 (8 mem.) & $12/40$ & 350--500 & $5\times10^{-3}$ & 512 \\
    \bottomrule
  \end{tabular}
  \caption{Hyperparameters for each experiment. $n$ is the ODE/PDE order, ``Grid'' the KAN grid size, and $m/\tau$ the number of delays and delay step (in samples); all KAN edges use spline order $3$ and the Adam optimizer. Dashes mark inapplicable fields: the sine-Gordon WSINDy and tanh-damping benchmarks and Korteweg--de Vries use the analytic jet, so none has a delay embedding. For Bouc--Wen, the state and force probes use grid $12$ at $350$ epochs, the motion-law read-out $400$ epochs, and the memory-law read-out grid $8$ at $500$ epochs.
  }
  \label{tab:supp-hyperparams}
\end{table}

\clearpage
\section{Learning a nonlinear derivative term}
\label{sec:supp-tanh}

The damped sine-Gordon of Section~\ref{sec:res-pde} of the main text has linear dissipation, so its $u_t$ edge is recovered as a slope and is indistinguishable in kind from a scalar coefficient. To test the derivative axis itself, we replace it with a saturating law,
\begin{equation}
  u_{tt} \;=\; u_{xx} \;-\; \sin(u) \;-\; \kappa\,\tanh(\nu\,u_t),
  \qquad \kappa = 0.15,\; \nu = 1.5,
  \label{eq:tanh-sg}
\end{equation}
evaluated on a breather, whose oscillation sweeps $u_t$ through both signs so that an odd damping law is identifiable. The equation layer is unchanged: free-form
  edges over $\{u,\,u_{xx},\,u_t\}$ (a mixed feature set where $u_{xx}$ is a spatial derivative), giving
\begin{equation}
  \hat{u}_{tt} \;=\; R(u) \;+\; \varphi_2(u_{xx}) \;+\;\varphi_3(u_t).
  \label{eq:tanh-edges}
\end{equation}
All three terms are recovered simultaneously from one noisy field: $R(u)$ matches $-\sin(u)$ to $0.0020$ RMS at zero noise and $0.0068$ at $2\%$; $\varphi_2$ is linear with slope $1.001$ (true $+1$); and $\varphi_3$ matches $-\kappa\tanh(\nu u_t)$ to $0.0011$ and $0.0036$ respectively, against a term whose own RMS amplitude is $0.1115$---roughly $1\%$ relative error at zero noise and $3\%$ at $2\%$ noise. Nothing was named in advance.

D-CIPHER~\citep{kacprzyk2023dcipher} searches for a relation
\begin{equation}
  \sum_{j} \eta_j\, \mathcal{O}_j u \;=\; g(\mathbf{x},u),
  \qquad \eta_j \in \mathbb{R},
  \label{eq:dcipher-form}
\end{equation}
in which the $\mathcal{O}_j$ are differential operators and only the source $g$ is discovered in closed form. The derivative dependence is therefore a linear combination with scalar weights, and $g$ may not depend on $u_t$; no choice of $\{\eta_j,g\}$ reproduces $\kappa\tanh(\nu u_t)$, so \eqref{eq:tanh-sg} lies outside the model class. The same holds for SINDy and WSINDy unless powers of $u_t$ are placed in $\Theta$ by hand, and a truncated polynomial column set cannot represent a saturating function. To quantify the gap without disadvantaging that family we compute its ceiling---the best constant-coefficient damping $\eta\,u_t$ fitted to the same learned edge---which is $0.0238$. Higher-Order-KANDy is thus $21\times$ closer to the true damping law at zero noise and still $6.6\times$ closer at $2\%$, and the ceiling is flat in the noise level because it is a structural limit of the model class rather than a statistical one.

%\clearpage
\section{Hysteresis and hidden memory variables (Bouc--Wen)}
\label{sec:supp-boucwen}

By the Takens embedding theorem, the delay lift reconstructs the full state from a scalar observable, including internal variables that are not derivatives of what we measure (Section~\ref{sec:delay}). We test this directly on the Bouc--Wen hysteretic oscillator, a standard model in structural dynamics whose restoring force depends on a hidden memory variable $z$ with its own nonlinear, non-smooth evolution law:
\begin{align}
  \ddot{x} &= -\alpha k\,x - c\,\dot{x} - (1-\alpha)k\,z, \label{eq:bw-motion}\\
  \dot{z}  &= A\,\dot{x} - \beta\,|\dot{x}|\,|z|^{n-1}z - \gamma\,\dot{x}\,|z|^{n}, \label{eq:bw-memory}
\end{align}
with unit mass, $c=0.15$, $k=1$, $\alpha=0.1$ (hysteresis-dominated, so $90\%$ of the restoring force is carried by $z$), and $A=\beta=\gamma=1$, $n=3$. We observe only $x(t)$, from an ensemble of free-decay trajectories of differing amplitude. Because the restoring force traces a hysteresis loop, it is multivalued in $x$---and in the instantaneous pair $(x,\dot{x})$, the state a first-order-reduced model sees (Fig.~\ref{fig:boucwen}a). It is not multivalued in the full jet: by Eq.~\eqref{eq:bw-motion}, $z=-(\ddot{x}+\alpha k x+c\dot{x})/((1-\alpha)k)$ is affine in $(x,\dot{x},\ddot{x})$. Reaching it that way requires a clean second derivative and a third-order closure $\dddot{x}=F(x,\dot{x},\ddot{x})$ whose right-hand side, obtained by substituting this $z$ into Eq.~\eqref{eq:bw-memory}, contains $|\dot{x}|\,|z|^{2}z$ and is not additively separable; a library-free equation layer cannot represent it. The delay window supplies $z$ without either.

This is precisely the regime the delay lift is built for, and it sharply separates the two front ends\rev{; to isolate the coordinates rather than the model, the same KAN layer is fit from each}. Handed the exact instantaneous coordinates $(x,\dot{x})$---the idealized jet front end, and the best case for SINDy---it recovers the hidden state only partially ($R^2=0.88$) and the restoring force and acceleration barely at all ($R^2=0.23$ and $0.24$), because both are dominated by the load reversals where $z$ is genuinely multivalued in $(x,\dot{x})$. The delay window resolves the ambiguity: it reconstructs $z$ almost exactly ($R^2=1.00$; Fig.~\ref{fig:boucwen}b), and predicts the restoring force and acceleration at $R^2=0.99$ (Supplementary Fig.~\ref{fig:boucwen}c).

The motion law~\eqref{eq:bw-motion} is linear in $(x,\dot{x},z)$. With the reconstructed state the edges recover $\ddot{x}=-0.100\,x-0.151\,\dot{x}-0.900\,z$---i.e.\ $\alpha k=0.10$, $c=0.15$, and coupling $(1-\alpha)k=0.90$, hence $k=1.00$ and $\alpha=0.10$, every term within a percent of truth. Replacing the true $z$ by the delay-reconstructed $\hat{z}$ (an observable-only pipeline: delays $\to\hat{z}\to$ equation layer) gives $-0.107\,x-0.170\,\dot{x}-0.906\,z$: the coupling stays within a percent, while the stiffness reads $7\%$ and the damping $13\%$ high, the cost of the lift's residual error concentrating on the smallest term. The memory law~\eqref{eq:bw-memory} is nonlinear and non-smooth; because a KAN layer sums univariate edges, its two cross terms enter as equation-layer features $|\dot{x}|z^3$ and $\dot{x}|z|^3$, and---given the lifted state, which the delays recover at $R^2=0.998$---the edges recover $\dot{z}=0.99\,\dot{x}-0.97\,|\dot{x}|z^3-1.01\,\dot{x}|z|^3$, i.e.\ $A,\beta,\gamma\approx1$ (Supplementary Table~\ref{tab:boucwen-coeffs}). Recovering this coupled, non-smooth hysteretic law from a single scalar channel is possible only because the delay lift first supplies the hidden state that the instantaneous pair $(x,\dot{x})$ cannot, and that the $2$-jet reaches only through a non-separable third-order closure.

The delay lift used above is a supervised probe: it is trained against the
true $z$, and on held-out data recovers it at $R^2=0.998$. Recovering the same
state without supervision is harder. A latent trained only to predict the
acceleration improves the held-out acceleration fit from $R^2=0.40$ (a comparable
instantaneous-coordinate baseline on that experiment's held-out split) to
$0.82\pm0.12$ across seeds, yet aligns with the true $z$ in shape
only (Pearson $0.79\pm0.11$; aligned $R^2\leq 0$). The coefficient read-outs
reported here therefore condition on the reconstructed state, and closing the gap
between supervised probes and unsupervised latent recovery remains open.

%\clearpage
\section{Undersampling and the compositional delay-to-jet map}
\label{sec:supp-undersampled}

The lifting layer maps a delay window to derivative estimates, and the question of how much architecture that map requires has a sharp answer set by the sampling rate rather than by noise. When the observable is densely sampled, differentiation is a local, linear operation: an optimal linear map on the window---a fixed weighted sum of neighboring samples, the class containing finite differences and Savitzky--Golay---is exact to the truncation order of Lemma~\ref{lem:stencil}, and nonlinearity gains nothing. When the data are undersampled, no Taylor-based finite-difference formula approximates the derivative---for nonlinear $f$, no fixed linear combination of window samples does (for linear $f$ a flow-adapted linear map can remain exact, Remark~\ref{rem:harmonic}); the window determines the state only through the inverse of the delay map $\Phi$, and the derivative is the Lie derivative of the observable evaluated at that state,
\begin{equation}
  x'(t) \;=\; (\mathcal{L}_f h)\big(\Phi^{-1}(\mathbf{y}_t)\big),
  \qquad
  \mathbf{y}_t = \big(x(t),\, x(t-\tau),\, \dots,\, x(t-(m-1)\tau)\big),
  \label{eq:comp-map}
\end{equation}
a genuine composition. For nonlinear $f$ this map is generically not additively separable, so a single KAN layer---a sum of univariate edges $\sum_i \phi_i(x(t-i\tau))$---is structurally too small a class, while the two-layer Kolmogorov--Arnold composition contains it. 

Additionally, combining results from the Takens' delay embedding theorem and the Kolmogorov-- Arnold representation theorem, we show the existence of a map that embeds the attractor in the differential (when the map is injective) or in a jet embedding.
This section first makes these arguments precise, then measures the gap directly.

\paragraph{Theoretical basis.}
Let $\varphi^{t}$ be the flow of $\dot{s}=f(s)$ on a compact invariant set
$\mathcal{A}$ of box-counting dimension $d$, let $h$ be the scalar observable,
$x(t)=h(s(t))$, and let
$$\Phi(s)=\big(h(s),\,h(\varphi^{-\tau}s),\,\dots,\,h(\varphi^{-(m-1)\tau}s)\big)$$
be the delay map, so that $\mathbf{y}_t=\Phi(s(t))$. For generic $(h,\tau)$ and
$m\geq 2d+1$, $\Phi$ is an embedding of $\mathcal{A}$, and a diffeomorphism onto
its image~\citep{takens1981, sauer1991embedology}.

\begin{proposition}[The jet identifiability at any sampling rate]
\label{prop:jet-ident}
Assume $f,h\in C^{k+1}$ and that $\Phi$ is an embedding. 
Then there is a continuous (and smooth when $\mathcal{A}$ is a smooth invariant manifold) function $G_k:\Phi(\mathcal{A})\to\mathbb{R}$ with $x^{(k)}(t)=G_k(\mathbf{y}_t)$ for all $t$, namely $G_k=(\mathcal{L}_f^{\,k}h)\circ\Phi^{-1}$, where $\mathcal{L}_f$ is the Lie derivative along $f$. 
No smallness condition on $\tau$ enters: the delay step is constrained only by the genericity of the embedding.
\end{proposition}
\begin{proof}
$x(t+\sigma)=h\big(\varphi^{\sigma}(s(t))\big)$; differentiating $k$ times at
$\sigma=0$ gives $x^{(k)}(t)=(\mathcal{L}_f^{\,k}h)(s(t))$, a $C^{1}$ function
of the state. 
Since $\Phi$ is an embedding, $s(t)=\Phi^{-1}(\mathbf{y}_t)$ with
$\Phi^{-1}$ smooth on the compact image; compose.
\end{proof}

\begin{corollary}[The jet embedding is a change of chart]
\label{cor:jet-chart}
For generic $(h,f)$ and $p\geq 2d$, the derivative map
$$\Theta(s)=\big(h(s),\,(\mathcal{L}_f h)(s),\,\dots,\,(\mathcal{L}_f^{\,p}h)(s)\big)$$
is itself an embedding of $\mathcal{A}$ and is the classical differential embedding, and the observability map of the system~\citep{takens1981, packard1980, Whalen2015Observability}. 
The full lift $G=(G_0,\dots,G_p)=\Theta\circ\Phi^{-1}$ of Proposition~\ref{prop:jet-ident} is therefore the transition map between two generic charts of the same attractor: the delay chart, which is what a scalar sensor measures, and the jet chart, in which a system that is $n$th-order in the observable takes the graph form $x^{(n)}=F(x,\dots,x^{(n-1)})$ of Eq.~\eqref{eq:nth-order}. 
The lifting layer thus constructs a new embedding of the data, and the map it must represent is a diffeomorphism between embedded copies of $\mathcal{A}$---a generic composition, with no reason to be additively separable in the delay coordinates.
\end{corollary}

\begin{proposition}[Diffeomorphism, conditioning, and stability]
\label{prop:jet-stability}
Assume, in addition to the hypotheses of Proposition~\ref{prop:jet-ident}, that $\mathcal{A}$ is a compact smooth invariant manifold with the ambient-induced metric and that
$\Theta=(h,\mathcal{L}_f h,\dots,\mathcal{L}_f^{\,p}h)$ is an embedding of $\mathcal{A}$. Write $F=\Theta\circ\Phi^{-1}$ and $\mathbf{z}=(z_0,\dots,z_p)$ for the jet coordinates. 
Then:
\begin{enumerate}[label=(\roman*)]
\item  $F:\Phi(\mathcal{A})\to\Theta(\mathcal{A})$ is a diffeomorphism;
its differential between attractor tangent spaces is
$DF_{\Phi(s)}=D\Theta_s\circ(D\Phi_s)^{-1}$, and $F$ conjugates the two
coordinate flows,
$F\circ\big(\Phi\,\varphi^{t}\,\Phi^{-1}\big)
 =\big(\Theta\,\varphi^{t}\,\Theta^{-1}\big)\circ F$.
\item In the jet chart the dynamics take the companion form
$\dot z_0=z_1,\;\dots,\;\dot z_{p-1}=z_p,\;\dot z_p=G(\mathbf{z})$ with the
single closure $G=(\mathcal{L}_f^{\,p+1}h)\circ\Theta^{-1}$---the graph form
of Eq.~\eqref{eq:nth-order} that the equation layer reads: in this chart the
entire vector field reduces to one unknown function.
\item  Let
$\alpha_\Delta=\inf_{s\in\mathcal{A}}
 \sigma_{\min}\!\big(D\Phi_s|_{T_s\mathcal{A}}\big)$ and
$M=\sup_{s\in\mathcal{A}}\|D\Theta_s\|$. Then $\alpha_\Delta>0$ and
$\sup_{\mathbf{y}\in\Phi(\mathcal{A})}
 \big\|DF_{\mathbf{y}}|_{T_{\mathbf{y}}\Phi(\mathcal{A})}\big\|
 \le M/\alpha_\Delta$.
\item If a learned map $\widehat F$ satisfies
$\|\widehat F-F\circ\pi\|_\infty\le\varepsilon$ on a neighborhood of
$\Phi(\mathcal{A})$ of radius below its reach, with $\pi$ the nearest-point
projection, then for a measured delay vector
$\tilde{\mathbf{y}}=\Phi(s)+\eta$ with $\|\eta\|$ below that radius,
\begin{equation}
  \|\widehat F(\tilde{\mathbf{y}})-\Theta(s)\|
  \;\le\;
  \varepsilon+c\,\frac{M}{\alpha_\Delta}\,\|\eta\|,
  \label{eq:jet-stability}
\end{equation}
with $c$ a curvature constant of the image ($c=1$ for perturbations tangent
to $\Phi(\mathcal{A})$).
\end{enumerate}
\end{proposition}
\begin{proof}
We prove the claims in turn.
\begin{enumerate}[label=(\roman*)]
\item Compositions and inverses of diffeomorphisms onto images; the conjugacy
is immediate from $F=\Theta\circ\Phi^{-1}$. 
\item $\dot z_i=\tfrac{d}{dt}\,\mathcal{L}_f^{\,i}h=\mathcal{L}_f^{\,i+1}h
 =z_{i+1}$ for $i<p$, and
$\dot z_p=(\mathcal{L}_f^{\,p+1}h)(s)=G(\mathbf{z})$ since $\Theta$ inverts.
\item $\sigma_{\min}\big(D\Phi|_{T\mathcal{A}}\big)$ is continuous and
positive on compact $\mathcal{A}$ because $\Phi$ immerses; the bound is the
chain rule with $\|(D\Phi)^{-1}\|=1/\sigma_{\min}(D\Phi)$. 
\item Triangle inequality, then the mean-value bound of (iii) along a path in $\Phi(\mathcal{A})$ from $\Phi(s)$ to $\pi(\tilde{\mathbf{y}})$, whose length is at most $c\,\|\eta\|$ within the reach.
\end{enumerate}
\end{proof}

The sampling interval enters Eq.~\eqref{eq:jet-stability} only through the
conditioning $\alpha_\Delta$ of the delay embedding---there is no
finite-difference truncation factor $\Delta^{-k}$ anywhere in the bound.
Coarse sampling therefore degrades reconstruction only insofar as it
degrades the observability of the delay map, and the failure criterion is
$\alpha_\Delta\to0$ (or loss of injectivity), not large $\Delta$ per se.

\begin{remark}[An exact delay-to-derivative map: the harmonic oscillator]
\label{rem:harmonic}
For $\ddot x=-\omega^2 x$ observed through $x$, two delays suffice: from
$x(t-\Delta)=x(t)\cos\omega\Delta-(\dot x(t)/\omega)\sin\omega\Delta$,
\begin{equation}
  \dot x(t)
  \;=\;
  \omega\,\frac{x(t)\cos\omega\Delta-x(t-\Delta)}{\sin\omega\Delta},
  \label{eq:harmonic-exact}
\end{equation}
exactly, at every $\Delta$ with $\sin\omega\Delta\neq0$. This is not
a finite difference: its weights are flow-adapted, its error does not grow
with $\Delta$, and it fails precisely at the resonances
$\omega\Delta=k\pi$, where the delay pair ceases to identify the state and
$\alpha_\Delta\to0$. Undersampling harms reconstruction through
observability, not truncation. (Here $m=2$ sits below the generic Takens
bound; the theorem's conditions are sufficient, not necessary.)
\end{remark}

\begin{proposition}[Representation classes]
\label{prop:rep-classes}
Let $G_k$ be as in Proposition~\ref{prop:jet-ident}. Then:
\begin{enumerate}[label=(\roman*)]
\item By the Kolmogorov superposition theorem \citep{Kolmogorov1957, Arnold2009_Kolmogorov}, every continuous function on a compact $K\subset\mathbb{R}^m$ can be written $\sum_{q=1}^{2m+1}\Phi_q\big(\sum_{p=1}^{m}\phi_{qp}(y_p)\big)$---the two-layer architecture with inner width $2m+1$. 
The classical inner functions are only continuous; smooth spline parameterizations of both layers are instead dense in $C(K)$~\citep{lai2021kolmogorov}, so the inner width is a capacity hyperparameter rather than the theorem's $2m+1$ (the probes below use width $12$ for $m=7$). 
Hence every $G_k$ of Proposition~\ref{prop:jet-ident} lies in the closure of the compositional (two-layer) lift class. 
\item The one-layer class is the additive family $\{\sum_p \phi_p(y_p)\}$: every $C^{2}$ member has identically vanishing mixed partials $\partial^2 F/\partial y_p\,\partial y_q$ ($p\neq q$). The composition $G_k=(\mathcal{L}_f^{\,k}h)\circ\Phi^{-1}$ lies in this class only when all its mixed partials vanish identically on $\Phi(\mathcal{A})$, which for nonlinear $f$ fails generically (it can hold in special cases: a scalar $\dot{x}=f(x)$ with $h=x$ gives $G_1(\mathbf{y})=f(y_1)$, and for linear $f$  Remark~\ref{rem:harmonic} gives an exactly linear $G_k$); Lemma~\ref{lem:mixed} below computes the mixed partial in closed form for $\ddot{x}=g(x)$. 
The additive class is therefore a proper subclass; how far a given $G_k$ lies from it is a per-system quantity, measured exactly in Table~\ref{tab:undersampled-jet} (Remark~\ref{rem:certificates}).
\item For the compositional (two-layer) form $F_k(\mathbf{y})=\sum_q\Psi_{kq}\big(\sum_i\phi_{qi}(y_i)\big)$, the Jacobian factorizes symbolically,
\begin{equation}
  DF(\mathbf{y})=A(\mathbf{y})\,B(\mathbf{y}),\qquad
  A_{kq}=\Psi'_{kq}\big(\textstyle\sum_i\phi_{qi}(y_i)\big),\qquad
  B_{qi}=\phi'_{qi}(y_i),
  \label{eq:kan-jacobian}
\end{equation}
so the differential of the learned multivariate chart change is expressed entirely through derivatives of its univariate edges. With $U_{\mathbf{y}}$ an orthonormal basis of the tangent space $T_{\mathbf{y}}\Phi(\mathcal{A})$---estimable from data alone by local principal components---local invertibility of the learned map on the attractor is certified by $\sigma_{\min}\big(A(\mathbf{y})B(\mathbf{y})U_{\mathbf{y}}\big)>0$, and its noise amplification by the tangent condition number $\kappa_F(\mathbf{y})=\sigma_{\max}/\sigma_{\min}$: both are computable from the trained network with no ground truth.
\end{enumerate}
\end{proposition}

The genericity in (ii) can be made exact in the simplest higher-order case,
where the departure from additivity is computable in closed form and is
proportional to the product of the nonlinearity and the sampling interval.

\begin{lemma}[Non-separability of the window-to-velocity map]
\label{lem:mixed}
Let $\ddot{x}=g(x)$ with $g\in C^{5}(\mathbb R)$, write
$\varphi^{s}(x,v)=\bigl(X(s;x,v),\,V(s;x,v)\bigr)$ for its phase-space
flow, and let $K\subset\mathbb R^{2}$ be a compact set of states
$(x,v)$ with nonempty interior. For $\Delta>0$ define the two-sample
delay map
$$
  H_\Delta(x,v)=\bigl(x,\;X(-\Delta;x,v)\bigr).
$$
Then there is $\Delta_0>0$ such that for every $\Delta\in(0,\Delta_0]$:
\begin{enumerate}[label=(\roman*)]
\item $H_\Delta$ is a $C^{2}$ diffeomorphism from a neighborhood of $K$
      onto its image, so the window-to-velocity map
      $G_{1,\Delta}=\pi_v\circ H_\Delta^{-1}$ is well defined and $C^{2}$
      on the open set $H_\Delta(\operatorname{int}K)$;
\item uniformly for $\mathbf y=(y_1,y_2)\in H_\Delta(K)$,
      \begin{equation}
        \frac{\partial^{2} G_{1,\Delta}}{\partial y_1\,\partial y_2}(\mathbf y)
        =\frac{\Delta}{12}\,g''(y_1)+O(\Delta^{2});
        \label{eq:mixed}
      \end{equation}
\item if $U\subset\operatorname{int}K$ is open and $g''$ has no zero on
      the $x$-projection of $U$, then, after shrinking $\Delta_0$,
      $G_{1,\Delta}$ is not additively separable on $H_\Delta(U)$: no
      functions $\phi_1,\phi_2$ satisfy
      $G_{1,\Delta}(\mathbf y)=\phi_1(y_1)+\phi_2(y_2)$ there;
\item if $g$ is affine, then $G_{1,\Delta}$ is exactly affine, hence
      additively separable, on $H_\Delta(K)$;
\item $G_{1,\Delta}$ admits an additive approximation on $H_\Delta(K)$
      with sup-norm error $O(\Delta^{2})$, namely
      $\mathbf y\mapsto (y_1-y_2)/\Delta+\tfrac{\Delta}{2}g(y_1)$.
\end{enumerate}
\end{lemma}

\begin{proof}
\emph{Step 1: expansion of the delay map.}
Along a solution, $\ddot x=g$, $\dddot x=g'\dot x$,
$x^{(4)}=g''\dot x^{2}+g'g$, and
$x^{(5)}=g'''\dot x^{3}+3g''g\,\dot x+(g')^{2}\dot x$. Taylor's theorem
in $s$ with integral remainder gives
\begin{equation}
\begin{aligned}
X(-\Delta;x,v)
  &= x-\Delta\,Q(x,v,\Delta),\\
Q
  &= v-\frac{\Delta}{2}g(x)+\frac{\Delta^{2}}{6}g'(x)\,v \\
  &\quad-\frac{\Delta^{3}}{24}
    \bigl(g''(x)v^{2}+g'(x)g(x)\bigr)
    +\Delta^{4}S(x,v,\Delta)
\end{aligned}
\label{eq:Q}
\end{equation}
where
$S(x,v,\Delta)=\int_0^1\frac{(1-\theta)^{4}}{4!}\,
 P\bigl(\varphi^{-\theta\Delta}(x,v)\bigr)\,d\theta$
with $P(x,v)=g'''(x)v^{3}+3g''(x)g(x)v+g'(x)^{2}v$.
Since $g\in C^{5}$, $P\in C^{2}(\mathbb R^{2})$ and the flow is $C^{5}$,
so $S$ and its first two derivatives in $(x,v)$ are bounded uniformly
on $K'\times[0,\Delta_0]$ for any compact neighborhood $K'$ of $K$.

\emph{Step 2: $H_\Delta$ is a diffeomorphism onto its image.}
From \eqref{eq:Q},
$\partial_v Q=1+\tfrac{\Delta^{2}}{6}g'(x)-\tfrac{\Delta^{3}}{12}g''(x)v
 +\Delta^{4}\partial_v S=1+O(\Delta^{2})$
uniformly on $K'$, so for $\Delta_0$ small $\partial_v Q\ge\tfrac12$ on
$K'\times(0,\Delta_0]$. Hence for each fixed $x$ the map
$v\mapsto X(-\Delta;x,v)=x-\Delta Q$ is strictly decreasing on the
$v$-slice of $K'$, which makes $H_\Delta$ injective on $K'$, and
$\det DH_\Delta=-\Delta\,\partial_v Q\neq0$. Thus $H_\Delta$ is a $C^{2}$
diffeomorphism of $K'$ onto its image, and $H_\Delta(\operatorname{int}K)$
is open. This proves (i).

\emph{Step 3: inversion.}
For $\mathbf y\in H_\Delta(K')$ put $q=(y_1-y_2)/\Delta$; then
$q=Q(y_1,v,\Delta)$ with $v=G_{1,\Delta}(\mathbf y)$. Because
$\partial_v Q\ge\tfrac12$, the implicit function theorem (with $\Delta$
as a parameter) yields $v=\mathcal V(y_1,q,\Delta)$, $C^{2}$ in
$(y_1,q)$ with derivatives bounded uniformly in $\Delta\in(0,\Delta_0]$
on the compact set $\mathcal K=\{(y_1,q)\}$ induced by $K'$ (note
$q=v+O(\Delta)$, so $\mathcal K$ is bounded). Let
$$
  v_3(y_1,q,\Delta)
  =q+\frac{\Delta}{2}g(y_1)-\frac{\Delta^{2}}{6}g'(y_1)\,q
   +\frac{\Delta^{3}}{24}g''(y_1)\,q^{2}
   -\frac{\Delta^{3}}{24}g'(y_1)g(y_1)
$$
be the third fixed-point iterate of $v=q+\tfrac{\Delta}{2}g
-\tfrac{\Delta^{2}}{6}g'v+\tfrac{\Delta^{3}}{24}(g''v^{2}+g'g)
-\Delta^{4}S(y_1,v,\Delta)$, and set $w=\mathcal V-v_3$. Substituting
$\mathcal V=v_3+w$ into this equation and collecting terms gives
$$
  \Bigl(1+\frac{\Delta^{2}}{6}g'(y_1)
        -\frac{\Delta^{3}}{24}g''(y_1)(\mathcal V+v_3)\Bigr)\,w
  =\Delta^{4}\,T(y_1,q,\Delta),
$$
where $T$ collects $S(y_1,\mathcal V,\Delta)$ together with the
$O(\Delta^{4})$ terms produced by the truncation of $v_3$; $T$ is a
composition of $C^{2}$ functions with uniform bounds, so $T$ and its
first two $(y_1,q)$-derivatives are uniformly bounded. Dividing by the
bracket, which is $1+O(\Delta^{2})$, yields
\begin{equation}
  G_{1,\Delta}(\mathbf y)
  =q+\frac{\Delta}{2}g(y_1)-\frac{\Delta^{2}}{6}g'(y_1)\,q
   +\frac{\Delta^{3}}{24}g''(y_1)\,q^{2}
   -\frac{\Delta^{3}}{24}g'(y_1)g(y_1)
   +\Delta^{4}R(y_1,q,\Delta),
  \label{eq:inverse-expansion}
\end{equation}
with $R\in C^{2}$ in $(y_1,q)$ and $R$, $\partial R$, $\partial^{2}R$
uniformly bounded on $\mathcal K\times(0,\Delta_0]$. Rewriting in
$(y_1,y_2)$,
\begin{align}
  G_{1,\Delta}(\mathbf y)
  ={}&\frac{y_1-y_2}{\Delta}+\frac{\Delta}{2}g(y_1)
      -\frac{\Delta}{6}g'(y_1)(y_1-y_2)
      +\frac{\Delta}{24}g''(y_1)(y_1-y_2)^{2}
      \notag\\
     &-\frac{\Delta^{3}}{24}g'(y_1)g(y_1)
      +\Delta^{4}R\!\Bigl(y_1,\frac{y_1-y_2}{\Delta},\Delta\Bigr).
  \label{eq:y-expansion}
\end{align}

\emph{Step 4: the mixed partial.}
The first, second, and fifth terms of \eqref{eq:y-expansion} are
additive and contribute nothing to $\partial^{2}_{y_1y_2}$. The third
and fourth contribute
$$
  \frac{\Delta}{6}g''(y_1)
  \quad\text{and}\quad
  -\frac{\Delta}{12}g''(y_1)-\frac{\Delta}{12}g'''(y_1)(y_1-y_2),
$$
respectively. Since $K$ has bounded velocity,
$y_1-y_2=\Delta q=O(\Delta)$ uniformly on $H_\Delta(K)$, so the last
piece is $O(\Delta^{2})$. For the remainder, the chain rule gives
$$
  \partial^{2}_{y_1y_2}\Bigl[\Delta^{4}R\Bigl(y_1,\tfrac{y_1-y_2}{\Delta},\Delta\Bigr)\Bigr]
  =-\Delta^{3}\,\partial_{qy_1}R-\Delta^{2}\,\partial_{qq}R
  =O(\Delta^{2}),
$$
using the uniform bounds on $\partial^{2}R$. Summing,
$\partial^{2}_{y_1y_2}G_{1,\Delta}
 =\tfrac{\Delta}{6}g''-\tfrac{\Delta}{12}g''+O(\Delta^{2})
 =\tfrac{\Delta}{12}g''(y_1)+O(\Delta^{2})$ uniformly on $H_\Delta(K)$,
which is (ii).

\emph{Step 5: consequences.}
(iii) If $|g''|\ge c>0$ on the $x$-projection of an open
$U\subset\operatorname{int}K$, then by (ii) there is $\Delta_0'$ such
that $|\partial^{2}_{y_1y_2}G_{1,\Delta}|\ge c\Delta/24>0$ on the open
set $H_\Delta(U)$ for $\Delta\le\Delta_0'$. Any $C^{2}$ function of the
form $\phi_1(y_1)+\phi_2(y_2)$ has identically vanishing mixed partial,
and if $G_{1,\Delta}$ agreed with such a sum on an open set the sum
would be $C^{2}$ there; hence no additive representation exists on
$H_\Delta(U)$.
(iv) If $g(x)=ax+b$, the vector field is affine, so $\varphi^{s}$,
$H_\Delta$, and $H_\Delta^{-1}$ (on the image, where it exists by (i))
are affine, and $G_{1,\Delta}=\pi_v\circ H_\Delta^{-1}$ is affine.
(v) In \eqref{eq:inverse-expansion} the first two terms are additive in
$(y_1,y_2)$, while $q$, $g$ and its derivatives, and $R$ are uniformly
bounded on $\mathcal K$, so the remaining terms are uniformly
$O(\Delta^{2})$ on $H_\Delta(K)$.
\end{proof}

\begin{remark}
The requirement $\operatorname{int}K\neq\emptyset$ in (iii) is not technical. 
On a curve $\{y_2=h(y_1)\}$ every function is trivially additive ($\phi_1=G\circ(\mathrm{id},h)$, $\phi_2=0$), so on the delay image of a single closed orbit of the conservative oscillator $\ddot x=g(x)$ the additive class always suffices, whatever $g$ is. 
The obstruction \eqref{eq:mixed} is visible only when the training data cover an open region of phase space---an ensemble of energies or transients---which is the regime in which the lifts of this paper are trained.
\end{remark}

On $\ddot{x}=-x^{3}$ ($g''=-6y_1$) the formula predicts $-\Delta\,y_1/2$; numerically inverting the flow and fitting $G_1$ locally gives $-0.049$ at $(y_1,\Delta)=(1,0.1)$ and $-0.197$ at $(1,0.4)$ against $-0.05$ and $-0.2$.
Lemma~\ref{lem:mixed} is the local form of the regime map: the same nonlinearity that closes the equation forces the lift out of the additive class, at a rate set by $\Delta$, and Lemma~\ref{lem:stencil} is the $\Delta\to0$ limit in which that departure disappears.

\begin{lemma}[The finite-difference regime]
\label{lem:stencil}
If $x\in C^{p+1}$ with $|x^{(p+1)}|\leq M$ on the window span, degree-$p$ polynomial interpolation on the window yields finite-difference weights $c^{(k)}_i$ with $\big|x^{(k)}(t)-\sum_i c^{(k)}_i\,x(t-i\tau)\big| \leq C_{p,m}\,M\, \tau^{\,p+1-k}$. 
Hence the distance from $G_k$ to the linear class---and a fortiori to the additive and compositional classes---is $O(\tau^{\,p+1-k})$ as $\tau\to 0$: at dense sampling all three classes are near-optimal, and nothing is gained from nonlinearity or depth.
\end{lemma}

\begin{definition}[Undersampled time series]
\label{def:undersampled}
A time series sampled at interval $\tau$ is undersampled for order $k$ when the truncation bound of Lemma~\ref{lem:stencil} is vacuous over any usable window---$\tau$ a non-negligible fraction of the dynamics' characteristic time---while the delay map remains an embedding with conditioning $\alpha_\Delta$ bounded away from zero (Proposition~\ref{prop:jet-stability}). 
In this regime $G_k$ exists and is stably invertible on the attractor, but only as a genuine composition: it is far from the linear and additive classes, and only the compositional class is guaranteed to contain it (Proposition~\ref{prop:rep-classes}). 
Reconstruction fails not when $\tau$ is large per se, but when $\alpha_\Delta\to0$ or the delay map loses injectivity (Remark~\ref{rem:harmonic}).
The \emph{aliasing wall} is the sampling interval at which the second condition fails: $\alpha_\Delta\to0$ or $\Phi$ ceases to be injective (Remark~\ref{rem:harmonic}).
\end{definition}

\begin{remark}[Certificates and scope]
\label{rem:certificates}
A general lower bound on the additive distance is delicate---$\Phi(\mathcal{A})$
is a $d$-dimensional submanifold of $\mathbb{R}^m$, so mixed-partial arguments
do not apply off the image---but it is measurable per system: with fixed spline
knots the additive model is linear in its parameters, its squared loss is
convex, and the fitted additive error in Table~\ref{tab:undersampled-jet} is
therefore (up to knot resolution and optimization tolerance) the class
distance rather than a training artifact. The linear rows are exact class
optima by construction (ridge regression). Three conventions and caveats
attach to the stability statements. Norms and singular values are taken with
per-order-normalized outputs and the ambient-induced metric on
$\mathcal{A}$; with heterogeneous derivative scales the conditioning claims
hold after this normalization, which the experiments apply throughout.
The perturbation bound~\eqref{eq:jet-stability} concerns the learned map on
a tubular neighborhood: ambient noise moves the delay vector off
$\Phi(\mathcal{A})$, so the off-manifold extension of $\widehat F$---shaped
in practice by training on noisy inputs, as in the field experiment
below---is part of the estimator, not an afterthought. Finally, the
propositions assume a smooth autonomous flow on a compact invariant
manifold: for fractal attractors the smooth-manifold hypotheses should be
replaced by the box-dimension embedding theory
of~\citet{sauer1991embedology}, under which $\Phi^{-1}$ is guaranteed
continuous rather than smooth; non-smooth dynamics (the Bouc--Wen memory
law) and forced systems (extend the state by the forcing phase) require the
correspondingly extended hypotheses.
\end{remark}

\paragraph{Existence and Kolmogorov--Arnold structure of the delay-to-jet map.}
Proposition~\ref{prop:jet-ident} and Corollary~\ref{cor:jet-chart} establish
that the lift exists and is a change of chart, and Lemma~\ref{lem:mixed}
shows that a one-layer (additive) KAN is generically too small to represent
it. The result below closes the argument from the other side, sharpening
Proposition~\ref{prop:rep-classes}(i). It shows that the delay-to-jet map
exists under injectivity of the delay map alone, that it is the transition
map between the delay and differential embeddings, that the
Kolmogorov--Arnold theorem makes this map, for all derivative orders at
once, exactly a two-layer KAN with a shared inner layer, and that two-layer spline KANs approximate it uniformly. 
The density statement is derived directly from the Kolmogorov--Arnold representation rather than assumed.

\medskip
\noindent\textit{Setting.}
Let $\varphi^{t}$ be the flow of $\dot s=f(s)$ on a manifold $M$ with
$f\in C^{p}$, let $\mathcal{A}\subset M$ be a compact invariant set of
box-counting dimension $d$, and let $h\in C^{p}(M,\mathbb{R})$ be the
observable, $x(t)=h(s(t))$. With the delay map $\Phi$ of
Proposition~\ref{prop:jet-ident} and the differential map
$\Theta(s)=\big(h(s),\,(\mathcal{L}_f h)(s),\,\dots,\,(\mathcal{L}_f^{\,p}h)(s)\big)$
of Corollary~\ref{cor:jet-chart}, we have
$x^{(k)}(t)=(\mathcal{L}_f^{\,k}h)(s(t))$, since
$\tfrac{d^k}{d\sigma^k}h(\varphi^{\sigma}s)\big|_{\sigma=0}
 =(\mathcal{L}_f^{\,k}h)(s)$. Write $\mathbf{y}_t=\Phi(s(t))$. We use two
hypotheses:
\begin{itemize}
  \item[(H1)] $\Phi$ is injective on $\mathcal{A}$. This holds for generic
    (prevalent) $(h,\tau)$ when $m\ge 2d+1$ for a smooth $\mathcal{A}$
    \citep{takens1981}, and when $m>2d$ for a fractal $\mathcal{A}$
    \citep{sauer1991embedology}.
  \item[(H2)] $\Theta$ is injective on $\mathcal{A}$ (generic for
    $p\ge 2d$, Corollary~\ref{cor:jet-chart}). Used only in part~(ii).
\end{itemize}
Unlike Proposition~\ref{prop:jet-ident}, part~(i) below requires only (H1),
not that $\Phi$ be an embedding, and so covers fractal attractors without
modification (cf.\ Remark~\ref{rem:certificates}).

\begin{proposition}[Kolmogorov--Arnold structure of the delay-to-jet map]
\label{prop:ka-delay-jet}
\leavevmode
\begin{enumerate}
  \item[(i)] \emph{Existence.} Under (H1) there is a unique continuous map
    $G=(G_0,\dots,G_p):\Phi(\mathcal{A})\to\mathbb{R}^{p+1}$ with
    \begin{equation}
      G\circ\Phi=\Theta,\qquad\text{i.e.}\qquad
      x^{(k)}(t)=G_k(\mathbf{y}_t)\quad\text{for all } t,\ k=0,\dots,p.
      \label{eq:ka-existence}
    \end{equation}
  \item[(ii)] \emph{Chart transition.} Under (H1) and (H2),
    $G:\Phi(\mathcal{A})\to\Theta(\mathcal{A})$ is a homeomorphism that
    conjugates the flow in delay coordinates to the flow in jet
    coordinates,
    \begin{equation}
      G\circ\bigl(\Phi\,\varphi^{t}\,\Phi^{-1}\bigr)
      =\bigl(\Theta\,\varphi^{t}\,\Theta^{-1}\bigr)\circ G .
      \label{eq:ka-conjugacy}
    \end{equation}
    If $\mathcal{A}$ is a smooth manifold and $\Phi,\Theta$ are
    embeddings, $G$ is a diffeomorphism, as in
    Proposition~\ref{prop:jet-stability}(i).
  \item[(iii)] \emph{Exact two-layer representation with a shared inner
    layer.} There exist a cube $K=[a,b]^m\supset\Phi(\mathcal{A})$ and
    continuous functions $\psi_{q,i}:[a,b]\to\mathbb{R}$ and
    $\Psi_{k,q}:\mathbb{R}\to\mathbb{R}$ such that for all
    $\mathbf{y}\in\Phi(\mathcal{A})$,
    \begin{equation}
      G_k(\mathbf{y})=\sum_{q=0}^{2m}\Psi_{k,q}\!\Bigl(\sum_{i=1}^{m}
        \psi_{q,i}(y_i)\Bigr),\qquad k=0,\dots,p .
      \label{eq:ka-exact}
    \end{equation}
    The inner functions $\psi_{q,i}$ do not depend on $k$; only the outer
    layer depends on the derivative order. The whole jet map is therefore
    a single two-layer KAN of shape $m\to 2m+1\to p+1$.
  \item[(iv)] \emph{Approximation by spline KANs.} For every
    $\varepsilon>0$ there is a two-layer KAN with spline edges and a
    shared inner layer,
    $\hat G_k(\mathbf{y})
     =\sum_{q}\hat\Psi_{k,q}\bigl(\sum_i\hat\psi_{q,i}(y_i)\bigr)$,
    such that
    \begin{equation}
      \sup_{\mathbf{y}\in\Phi(\mathcal{A})}\ \max_{0\le k\le p}\
      \bigl|\hat G_k(\mathbf{y})-G_k(\mathbf{y})\bigr|<\varepsilon .
      \label{eq:ka-approx}
    \end{equation}
\end{enumerate}
\end{proposition}

\begin{proof}
We prove each part in turn.
\begin{enumerate}[label=(\roman*)]
\item  By (H1), $\Phi:\mathcal{A}\to\Phi(\mathcal{A})$ is a continuous
bijection from a compact space onto a Hausdorff space, hence a
homeomorphism, and $\Phi^{-1}:\Phi(\mathcal{A})\to\mathcal{A}$ is
continuous. Set $G:=\Theta\circ\Phi^{-1}$. Each component
$G_k=(\mathcal{L}_f^{\,k}h)\circ\Phi^{-1}$ is continuous because
$\mathcal{L}_f^{\,k}h$ is continuous for $f,h\in C^{p}$ and $k\le p$, and
$$
  x^{(k)}(t)=(\mathcal{L}_f^{\,k}h)(s(t))
  =(\mathcal{L}_f^{\,k}h)\bigl(\Phi^{-1}(\mathbf{y}_t)\bigr)
  =G_k(\mathbf{y}_t).
$$
Any $G'$ with $G'\circ\Phi=\Theta$ equals $\Theta\circ\Phi^{-1}$ on
$\Phi(\mathcal{A})$, which gives uniqueness. No condition on $\tau$ enters
beyond (H1).

\item Under (H2), $\Theta$ is likewise a homeomorphism onto
$\Theta(\mathcal{A})$, so $G=\Theta\circ\Phi^{-1}$ is a homeomorphism with
inverse $\Phi\circ\Theta^{-1}$. The conjugacy follows by substitution:
$$
  G\,\Phi\,\varphi^{t}\,\Phi^{-1}
  =\Theta\,\varphi^{t}\,\Phi^{-1}
  =\bigl(\Theta\,\varphi^{t}\,\Theta^{-1}\bigr)\bigl(\Theta\,\Phi^{-1}\bigr)
  =\bigl(\Theta\,\varphi^{t}\,\Theta^{-1}\bigr)\,G .
$$
In the smooth case, compositions and inverses of embeddings onto their
images are diffeomorphisms.

\item \textit{Step 1: extension to a cube.} $\Phi(\mathcal{A})$ is
compact, so it lies in some cube $K=[a,b]^m$ and is closed in $K$. By the
Tietze extension theorem each $G_k$ extends to a continuous
$\tilde G_k:K\to\mathbb{R}$ agreeing with $G_k$ on $\Phi(\mathcal{A})$.
This step is needed because the Kolmogorov--Arnold theorem applies to
continuous functions on a cube, whereas $G_k$ is defined only on the
(typically lower-dimensional, possibly fractal) set $\Phi(\mathcal{A})$.
The extension is not unique, but every extension agrees with $G_k$
wherever data can occur.

\textit{Step 2: Kolmogorov's theorem.} After rescaling $K$ to $[0,1]^m$,
the Kolmogorov superposition theorem
\citep{Kolmogorov1957, Arnold2009_Kolmogorov}, in its form with universal
inner functions, provides continuous $\psi_{q,i}$ ($q=0,\dots,2m$,
$i=1,\dots,m$), depending only on $m$, such that every continuous $g$ on
$[0,1]^m$ admits
$g(\mathbf{y})=\sum_{q=0}^{2m}\chi_q\bigl(\sum_i\psi_{q,i}(y_i)\bigr)$
with continuous outer functions $\chi_q$ depending on $g$. Applying this to
each $\tilde G_k$ with the same $\psi_{q,i}$ yields outer functions
$\Psi_{k,q}$; restricting to $\Phi(\mathcal{A})$ gives
\eqref{eq:ka-exact}.

\item Let $n_h=2m+1$ and $s_q(\mathbf{y})=\sum_i\psi_{q,i}(y_i)$,
whose range over $K$ is a compact interval $I_q$; let $I_q^{1}$ denote
$I_q$ enlarged by $1$ on each side.

\textit{Tolerance for the inner layer.} Each $\Psi_{k,q}$ is uniformly
continuous on $I_q^{1}$, so there is $\delta\in(0,1)$ such that for all
$k,q$,
$$
  |u-v|<\delta\ \Longrightarrow\
  |\Psi_{k,q}(u)-\Psi_{k,q}(v)|<\frac{\varepsilon}{2n_h}.
$$

\textit{Inner layer.} Piecewise-linear interpolants, and hence splines of degree $d$ are dense in $C[a,b]$. 
Choose splines $\hat\psi_{q,i}$ with $\|\hat\psi_{q,i}-\psi_{q,i}\|_\infty<\delta/m$. 
Then $\hat s_q=\sum_i\hat\psi_{q,i}$ satisfies $|\hat s_q-s_q|<\delta$, so $\hat s_q(\mathbf{y})\in I_q^{1}$.

\textit{Outer layer.} Choose splines $\hat\Psi_{k,q}$ on $I_q^{1}$ with
$\|\hat\Psi_{k,q}-\Psi_{k,q}\|_{\infty,I_q^{1}}<\varepsilon/(2n_h)$.

\textit{Combination.} For $\mathbf{y}\in\Phi(\mathcal{A})$,
$$
  \bigl|\hat\Psi_{k,q}(\hat s_q)-\Psi_{k,q}(s_q)\bigr|
  \le\underbrace{\bigl|\hat\Psi_{k,q}(\hat s_q)-\Psi_{k,q}(\hat s_q)\bigr|}_{<\,\varepsilon/2n_h}
  +\underbrace{\bigl|\Psi_{k,q}(\hat s_q)-\Psi_{k,q}(s_q)\bigr|}_{<\,\varepsilon/2n_h}.
$$
Summing over the $n_h$ hidden units gives $|\hat G_k-G_k|<\varepsilon$. The
inner splines were chosen independently of $k$, so the inner layer remains
shared.
\end{enumerate}
\end{proof}

\begin{corollary}[Lift error propagates linearly to the discovered law]
\label{cor:ka-residual}
Suppose the system is $n$th-order in the observable, $n\le p$, so that
Eq.~\eqref{eq:nth-order} holds, $x^{(n)}=F\bigl(x,x',\dots,x^{(n-1)}\bigr)$,
with a continuous closure $F$ (namely
$F=(\mathcal{L}_f^{\,n}h)\circ\Theta_{n-1}^{-1}$ when the $(n-1)$-jet map
$\Theta_{n-1}$ is injective on $\mathcal{A}$; this is the closure of the
companion form in Proposition~\ref{prop:jet-stability}(ii)). If $F$ is
$L$-Lipschitz and the lift satisfies $\|\hat G-G\|_\infty<\varepsilon$ on
$\Phi(\mathcal{A})$, then along every trajectory
\begin{equation}
  \Bigl|\hat G_n(\mathbf{y}_t)
    -F\bigl(\hat G_0(\mathbf{y}_t),\dots,\hat G_{n-1}(\mathbf{y}_t)\bigr)\Bigr|
  \le\bigl(1+L\sqrt{n}\bigr)\,\varepsilon .
  \label{eq:ka-residual}
\end{equation}
\end{corollary}

\begin{proof}
Insert $G_n(\mathbf{y}_t)=F\bigl(G_0(\mathbf{y}_t),\dots,G_{n-1}(\mathbf{y}_t)\bigr)$
and apply the triangle inequality. 
The first term is bounded by $\varepsilon$; the second by $L\,\bigl\|(\hat G_k-G_k)_{k<n}\bigr\|_2\le L\sqrt{n}\,\varepsilon$.
\end{proof}

\begin{remark}[Interpretation]
\label{rem:ka-interpretation}
Proposition~\ref{prop:ka-delay-jet} supplies the logical chain behind the
lifting layer.
\emph{Existence} (i) makes the lift well-posed at any $\tau$ for which
(H1) holds; where (H1) fails, no architecture can succeed, which is the
content of the aliasing wall of Definition~\ref{def:undersampled}
(Table~\ref{tab:undersampled-jet} at $\Delta t=0.8$) and of the Bouc--Wen
memory variable seen through $(x,\dot x)$ (Note~\ref{sec:supp-boucwen}).
\emph{Chart transition} (ii) shows that the flow read in jet coordinates
is conjugate to the observed flow, and in those coordinates the vector
field reduces to the single closure of
Proposition~\ref{prop:jet-stability}(ii); this is what licenses reading
the equation layer's edges as the governing law.
\emph{Representation} (iii) shows that depth two is sufficient for every
system satisfying (H1), whatever the nonlinearity or sampling rate, and
that one inner layer can serve all $p+1$ derivative orders, which is the
jet operator of Eq.~\eqref{eq:jetop} in its compositional (two-layer)
form. 
% Lemma~\ref{lem:mixed} supplies the matching lower bound, that one
% layer is generically insufficient, so the two results together fix the depth at exactly two. 
Lemma~\ref{lem:mixed} supplies the complementary lower bound by showing
that a single additive layer can be insufficient. Thus, two layers are
sufficient in general and become necessary within the nonseparable regime
identified by the lemma.
Corollary~\ref{cor:ka-residual} connects lift
accuracy to the residual of the discovered law.
\end{remark}

\begin{remark}[Scope]
\label{rem:ka-scope}
The inner functions of the Kolmogorov--Arnold theorem are continuous in general, but far from smooth and often ``Cantor-like'' \citep{girosi1989representation}.
Part~(iii) is therefore an existence statement about architecture;
part~(iv) is what a trained spline KAN realizes, and it inherits no convergence rate from the theorem. 
Rates must come from the smoothness of $G$ (Proposition~\ref{prop:jet-stability}) and from spline approximation theory \citep{lai2021kolmogorov}. 
In particular, the width used in the experiments ($12$ for $m=7$) is a capacity choice, not the theorem's $2m+1$ with its pathological inner functions, as noted in Proposition~\ref{prop:rep-classes}(i).  
Corollary~\ref{cor:ka-residual} assumes a Lipschitz closure and exact targets; measurement noise enters separately through Proposition~\ref{prop:jet-stability}(iv).
\end{remark}

\paragraph{From representation classes to measurement.}
The results above establish which model classes can represent the lift, but not how large the gaps between them are in practice. Proposition~\ref{prop:ka-delay-jet} is an existence statement with no rate (Remark~S3.14), Lemma~\ref{lem:mixed} is asymptotic in $\Delta$, and the distance to the additive class can only be measured system by system (Remark~\ref{rem:certificates}). 
Taken together, however, they make three predictions that can be checked directly as the sampling interval grows:
\begin{enumerate}[label=(\roman*)]
\item at dense sampling, the linear class is already near-exact and nothing is
  gained from nonlinearity or depth (Lemma~\ref{lem:stencil});
\item in the undersampled regime, the linear and additive classes fall away
  while the compositional class remains accurate (Lemma~\ref{lem:mixed},
  Proposition~\ref{prop:ka-delay-jet}(iii));
\item beyond the aliasing wall, every class fails, because the window no longer
  determines the state (Definition~\ref{def:undersampled}).
\end{enumerate}
We test these on the Lorenz system observed through $x(t)$. 
Its strong nonlinearity makes the departure from additivity large, and its hidden states supply exact derivative targets. 
Along the same sweep, we evaluate the conditioning certificate of Proposition~\ref{prop:rep-classes}(iii). We then repeat the comparison on a coarsely sampled, noisy sine-Gordon field, where the delay window becomes a space--time patch.

\paragraph{Protocol.}
The Lorenz system is observed through the scalar $x(t)$ alone (hidden $y,z$),
noise-free, so the probe isolates representation capacity from estimation
error. Targets are exact, $x' = \sigma(y-x)$ and $x'' = \sigma(\dot y - x')$,
evaluated from the simulator's hidden states at the center of each window of
$m=7$ consecutive samples at spacing $\Delta t$. Three model classes fit the
same windows: an optimal linear map (ridge regression on the window), a
one-layer KAN (the additive class of the jet operator, Eq.~\eqref{eq:jetop}),
and a two-layer KAN of hidden width $12$. Training uses one trajectory;
evaluation uses an independent trajectory from a different initial condition,
both burned onto the attractor. KAN fits are repeated over three seeds and
reported as the mean; the linear map is deterministic. Sweeping
$\Delta t$ from $0.01$ to $0.80$---from dense sampling, through roughly three
samples per orbit turn at $0.25$, to less than one at $0.80$---moves the task
from the finite-difference regime, through the flow-inversion regime, to the aliasing
wall of Definition~\ref{def:undersampled}.

\begin{table}[t]
  \centering
  \begin{tabular}{llccc}
    \toprule
    $\Delta t$ & Target & linear & additive (1-layer) & compositional (2-layer) \\
    \midrule
    $0.01$ & $x'$  & $1.0000$ & $0.9849\;[0.9561,\,0.9996]$ & $1.0000\;[1.0000,\,1.0000]$ \\
           & $x''$ & $1.0000$ & $0.9721\;[0.9692,\,0.9737]$ & $0.9992\;[0.9991,\,0.9993]$ \\
    \midrule
    $0.05$ & $x'$  & $1.0000$ & $0.9998\;[0.9997,\,0.9999]$ & $1.0000\;[1.0000,\,1.0000]$ \\
           & $x''$ & $1.0000$ & $0.9986\;[0.9982,\,0.9988]$ & $1.0000\;[0.9999,\,1.0000]$ \\
    \midrule
    $0.15$ & $x'$  & $0.8967$ & $0.9712\;[0.9545,\,0.9851]$ & $0.9997\;[0.9997,\,0.9997]$ \\
           & $x''$ & $0.9279$ & $0.9756\;[0.9690,\,0.9877]$ & $0.9994\;[0.9992,\,0.9996]$ \\
    \midrule
    $0.25$ & $x'$  & $0.4436$ & $0.8986\;[0.8463,\,0.9326]$ & $\mathbf{0.9940}\;[0.9931,\,0.9949]$ \\
           & $x''$ & $0.6351$ & $0.9008\;[0.8512,\,0.9867]$ & $\mathbf{0.9963}\;[0.9953,\,0.9969]$ \\
    \midrule
    $0.40$ & $x'$  & $0.0642$ & $0.5718\;[0.5695,\,0.5741]$ & $\mathbf{0.8810}\;[0.8761,\,0.8852]$ \\
           & $x''$ & $0.3428$ & $0.9056\;[0.8682,\,0.9671]$ & $\mathbf{0.9699}\;[0.9675,\,0.9722]$ \\
    \midrule
    $0.60$ & $x'$  & $0.0026$ & $0.2062\;[0.1970,\,0.2126]$ & $\mathbf{0.6097}\;[0.5363,\,0.7012]$ \\
           & $x''$ & $0.2538$ & $0.9216\;[0.9004,\,0.9593]$ & $0.9315\;[0.9227,\,0.9362]$ \\
    \midrule
    $0.80$ & $x'$  & $-0.0266$ & $-0.2900\;[-0.3131,\,-0.2621]$ & $0.2132\;[0.1019,\,0.3844]$ \\
           & $x''$ & $0.2958$ & $0.8852\;[0.8632,\,0.9273]$ & $0.8675\;[0.8496,\,0.8864]$ \\
    \bottomrule
  \end{tabular}
  \caption{Held-out $R^2$ for the window-to-jet map on the Lorenz observable
    versus sampling interval $\Delta t$ ($m=7$, noiseless, independent test
    trajectory). KAN entries are seed means with $[\min,\max]$ over three
    seeds; the linear map is deterministic. Dense sampling: the linear
    map is exact and all classes tie. Coarse sampling: the linear map
    collapses, the additive (one-layer) class plateaus, and only the
    compositional (two-layer) map remains near-exact---and is also the most
    seed-stable.}
  \label{tab:undersampled-jet}
\end{table}

\paragraph{Results.}
Table~\ref{tab:undersampled-jet} shows three regimes. At $\Delta t \le 0.05$ the linear map is exact
($R^2 = 1.000$) and every class ties, confirming that with dense data the
lift gains nothing from nonlinearity---differentiation is linear there, and
the one-layer jet of Eq.~\eqref{eq:jetop} is the appropriate model. At
$\Delta t = 0.25$ the ordering is graded and large: the linear map collapses,
the additive class improves on it but plateaus, and the compositional map
stays near-exact on both derivative orders. In residual terms, at
$\Delta t = 0.25$ the two-layer map reduces $1-R^2$ on $x''$ from $0.365$
(linear) and $0.099$ (additive) to $0.0037$---a factor of $27$ over the
additive class and nearly $100$ over the linear map. The additive class beating
the optimal linear map shows that nonlinearity already helps; only
composition closes the gap, as Eq.~\eqref{eq:comp-map} predicts. Extending
the sweep locates the lower boundary: the compositional advantage persists
through $\Delta t=0.6$ ($x'$: $0.61$ versus $0.21$ additive and $0.00$
linear), and by $\Delta t=0.8$---below one sample per orbit turn---every
class fails on $x'$ (best $R^2=0.21$): the aliasing wall of
Definition~\ref{def:undersampled}, where no representation can recover
information the samples no longer carry.

\begin{table}[t]
  \centering
  \begin{tabular}{lcccc}
    \toprule
    $\Delta t$ & med.\ $\sigma_{\min}$ & p5 $\sigma_{\min}$
               & med.\ $\kappa_F$ & compositional $R^2$ ($x'$) \\
    \midrule
    $0.01$ & $0.62$ & $0.38$ & $7.3$ & $1.000$ \\
    $0.05$ & $0.62$ & $0.36$ & $1.6$ & $1.000$ \\
    $0.15$ & $0.41$ & $0.08$ & $2.5$ & $1.000$ \\
    $0.25$ & $0.39$ & $0.12$ & $3.2$ & $0.994$ \\
    $0.40$ & $0.51$ & $0.12$ & $4.0$ & $0.881$ \\
    $0.60$ & $0.63$ & $0.16$ & $4.5$ & $0.610$ \\
    $0.80$ & $0.73$ & $0.16$ & $4.2$ & $0.213$ \\
    \bottomrule
  \end{tabular}
  \caption{Tangent-space conditioning of the learned chart change
    (Proposition~\ref{prop:rep-classes}(iii)): median and 5th-percentile
    $\sigma_{\min}(J_F U_{\mathbf{y}})$ and median $\kappa_F$ over $300$
    held-out windows, seed means, per-order-normalized units, tangent basis
    from local PCA ($k=40$, $d=2$). $\kappa_F$ is U-shaped in $\Delta t$
    (delay redundancy at dense sampling, observability loss at coarse), while
    the invertibility certificate stays positive even where reconstruction
    fails.}
  \label{tab:jet-conditioning}
\end{table}

\paragraph{Conditioning of the learned chart change.}
Proposition~\ref{prop:rep-classes}(iii) makes the trained lift auditable. 
At $300$ held-out windows we assemble the joint Jacobian of the fitted chart change---the exact order-$0$ projection row stacked with the analytic gradients of the fitted $x'$ and $x''$ maps, in per-order-normalized units---restrict it to a tangent basis $U_{\mathbf{y}}$ estimated by local principal components, and report $\sigma_{\min}(J_F U_{\mathbf{y}})$ and $\kappa_F$ (Table~\ref{tab:jet-conditioning}). Two findings. 
First, $\kappa_F$ is U-shaped in $\Delta t$ with its minimum near $\Delta t=0.05$: at dense sampling the delay coordinates are nearly redundant---the classical too-short-delay degeneracy, $\alpha_\Delta$ small---while at coarse sampling observability degrades, so noise amplification is governed at both extremes by the conditioning of the delay embedding, the mechanism of Proposition~\ref{prop:jet-stability}(iii). 
Second, the invertibility certificate stays bounded away from zero even at $\Delta t=0.8$, where reconstruction of $x'$ nonetheless fails. 
Local nonsingularity of the learned map is therefore necessary but not sufficient: at the aliasing wall the failure is global---the true chart change loses injectivity and the approximation term of
Eq.~\eqref{eq:jet-stability} dominates---and a pointwise Jacobian audit cannot see it. 
%(The local-PCA tangent estimate itself also degrades once the image stops being locally two-dimensional.)

\paragraph{Scope.}
This is a supervised representation probe, in the same sense as the Bouc--Wen state probe of Note~\ref{sec:supp-boucwen}: targets use the simulator's hidden states, so the experiment measures what each class can represent. 
It identifies the two regimes in which a compositional lift is warranted---undersampled observations, where the window must invert the flow, and hidden non-derivative states, where the window must reconstruct a coordinate (Note~\ref{sec:supp-boucwen})---and the regime in which the single-layer jet is already the right tool, namely densely sampled observables, which covers the calibration systems examined.

\paragraph{Field-side discovery under joint undersampling and noise
(sine-Gordon).}
The same class separation governs field data, where the delay window becomes a space--time patch. Undersampling alone does not produce it: on the noiseless kink--antikink family, an interpolating tensor spline recovers $u_{tt}$ to NRMSE $0.015$ even from $33$ frames, because piecewise-quintic interpolation converges rapidly on an analytic field. 
Noise closes that escape: it forbids interpolation, and the smoothing bias--variance tradeoff then binds every linear front end. 
We therefore test discovery of $u_{tt}=c^2 u_{xx} + s\,\sin(u)$ (true $c^2=1$, $s=-1$) on a coarse, noisy field: $80$ spatial points ($\Delta x = 0.46$ against a kink width of $\approx 0.75$), $33$ frames ($\Delta t = 0.5$), and $1\%$ Gaussian noise.
Patch operators ($5\times5$ coarse samples $\to$ $(u, u_{xx}, u_{tt})$ at the patch center; one- versus two-layer KAN) are calibrated on exact solutions at pooled velocities $v\in\{0.40,0.45,0.50\}$ observed through the same coarse noisy grid, and deployed on the unseen velocity $v=0.55$ with fresh noise; exact solution values are used for scoring only. 
The classical comparator is the paper's own analytic-spline jet (Eq.~\eqref{eq:fieldjet}) fit to the coarse noisy field, with its smoothing parameter swept against ground truth on an oracle-tuned linear baseline.

\begin{table}[t]
  \centering
  \begin{tabular}{llcccc}
    \toprule
    Regime & Front end & NRMSE $u_{xx}$ & NRMSE $u_{tt}$ & $\hat{c}^2$ & $\hat{s}$ \\
    \midrule
    coarse, $1\%$ noise
      & analytic spline (oracle) & $0.280$ & $0.342$ & $0.691$ & $-0.514$ \\
      & additive (1-layer)       & $0.099$ & $0.283$ & $0.954$ & $-0.989$ \\
      & compositional (2-layer)  & $\mathbf{0.047}$ & $\mathbf{0.087}$
                                 & $\mathbf{0.985}$ & $\mathbf{-0.970}$ \\
    \midrule
    dense, $1\%$ noise
      & analytic spline (oracle) & $0.072$ & $0.147$ & $0.977$ & $-0.959$ \\
      & additive (1-layer)       & $0.304$ & $0.693$ & $0.636$ & $-0.547$ \\
      & compositional (2-layer)  & $0.135$ & $0.274$ & $0.981$ & $-0.973$ \\
    \bottomrule
  \end{tabular}
  \caption{Sine-Gordon derivative accuracy and discovered coefficients
    (least squares on $[u_{xx}, \sin u]$; true $(c^2,s)=(1,-1)$) on the
    held-out velocity, means over three seeds (noise realization and network
    initialization vary jointly). Coarse: $80\times33$ samples; dense:
    $400\times257$. The spline's smoothing is oracle-tuned per realization.}
  \label{tab:sg-undersampled}
\end{table}

\begin{figure}[t]
  \centering
  \includegraphics[width=\linewidth]{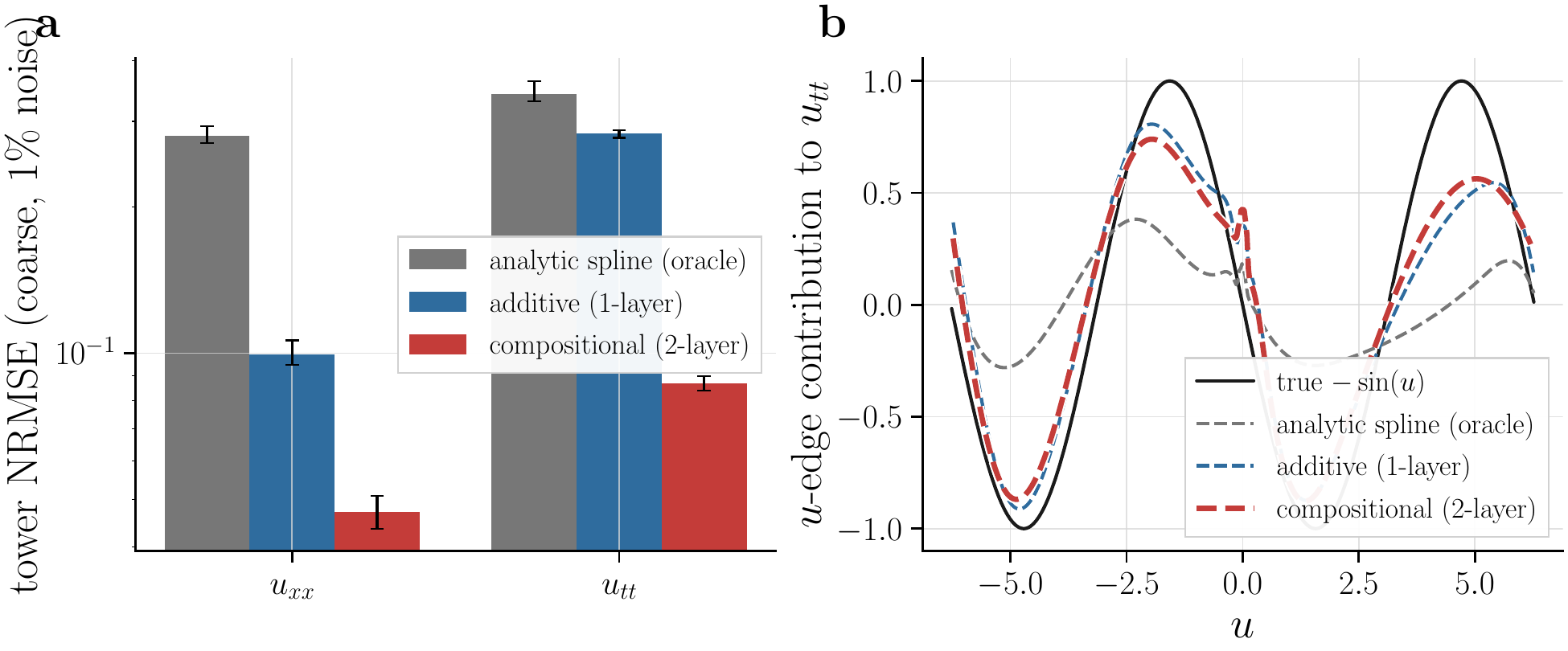}
  \caption{Coarse, noisy sine-Gordon. 
  (a) Derivative NRMSE per front end (bars: seed mean; whiskers: min--max).
  (b) The library-free $u$-edge of the equation layer against the true $-\sin(u)$, per front end (first seed).}
  \label{fig:sg-undersampled}
\end{figure}

Table~\ref{tab:sg-undersampled} and Fig.~\ref{fig:sg-undersampled} show three things. 
First, on the coarse noisy field the oracle-tuned spline fails discovery: errors-in-variables attenuation from its derivative estimates (NRMSE $0.28$--$0.34$) biases the recovered law by $31\%$ on $c^2$ and $49\%$ on $s$, and no smoothing choice repairs this. Both learned patch operators restore discovery; the compositional map has the cleanest derivatives ($u_{tt}$ NRMSE $0.087$, three times below the additive map and roughly four times below the oracle spline) and the tightest, most seed-stable coefficients. 
With the library-free equation layer the compositional front end recovers, on its best seed, $\hat{c}^2=1.00$ and $\hat{s}=-0.95$ with the $u$-edge matching $-\sin(u)$ at correlation $0.999$ (seed-mean correlation $0.980$; the spread comes from the equation-layer fit, whose collinear-edge variance is a back-end property independent of the front end). 
Second, the additive map survives on coefficients ($\hat{c}^2=0.95$) despite its three-fold worse derivative error: with dense-enough spatial context, $u_{tt}$ can be written as a fixed linear combination of the patch samples plus a univariate function of $u$, which is additive, and its residual errors decorrelate from the regressors. 
The separation between the classes is therefore carried by derivative fidelity and stability rather than by outright failure---weaker than the scalar case, as the additive route predicts. Third, the dense control flips the ordering: with $257$ frames the global spline is the best front end ($u_{tt}$ NRMSE $0.147$) and the small-patch operators are handicapped by their receptive field under noise (the additive map fails outright). 
The regime map runs in both directions---global linear smoothing is the right tool for dense noisy fields, and the compositional patch operator is the only front end that keeps discovery accurate on coarse noisy ones.

%\clearpage
\section{Controlled comparison on the Genesio--Tesi system: front ends across sampling interval}
\label{sec:supp-controlled}

The main-text benchmarks vary the derivative estimate and the equation representation together. This note separates them on the third-order Genesio--Tesi system $\dddot{x}=-a\ddot{x}-b\dot{x}-cx+qx^{2}$ ($a=1.2$, $b=2.92$, $c=6$, $q=1$), observed through $x$ alone, as the sampling interval $\Delta$ is coarsened from $0.05$ to $0.9$ (the attractor's half-period is ${\approx}1.9$, so $\Delta=0.9$ is about two samples per swing). The protocol has a dense calibration trajectory (one initial condition, $\Delta t=0.01$, $600$ time units) whose derivative targets are Savitzky--Golay estimates of the observable itself---no simulator state enters---and a coarse deployment trajectory from a different initial condition on which only samples at spacing $\Delta$ exist and discovery is run; exact derivatives are used for scoring only. Five front ends produce $(\hat{x}',\hat{x}'',\hat{x}''')$ on the deployment trajectory from windows of $m=7$ coarse samples: Savitzky--Golay on the coarse samples with window and polynomial order tuned per derivative against ground truth (an oracle-tuned fixed formula); the optimal linear map (ridge regression on the window); the additive class solved exactly (least squares on a per-coordinate cubic B-spline basis, no optimizer); an additive (one-layer) KAN trained by Adam; and a compositional (two-layer) KAN of hidden width $12$. The deployment trajectory is noise-free; the comparison isolates sampling interval, not noise. Three backends consume each front end's derivatives unchanged: least squares on the oracle library $[x,\dot{x},\ddot{x},x^{2}]$, STLSQ on the generic degree-two polynomial library in $(x,\dot{x},\ddot{x})$ (ten terms, threshold $0.3$), and the library-free KANDy equation layer over $(x,\dot{x},\ddot{x})$. KAN fits use three seeds; the linear and exact-additive maps are deterministic.

\begin{table}[h]
  \centering\small
  \begin{tabular}{lcccccc}
    \toprule
    & \multicolumn{6}{c}{sampling interval $\Delta$} \\
    front end & $0.05$ & $0.1$ & $0.2$ & $0.4$ & $0.6$ & $0.9$ \\
    \midrule
    \multicolumn{7}{l}{\emph{held-out $R^{2}$ of $\hat{x}'''$}} \\
    Savitzky--Golay (oracle-tuned) & $1.000$ & $1.000$ & $1.000$ & $0.996$ & $0.959$ & $0.788$ \\
    linear (ridge on window)       & $1.000$ & $1.000$ & $1.000$ & $1.000$ & $0.994$ & $0.879$ \\
    additive, exact LS             & $1.000$ & $1.000$ & $1.000$ & $1.000$ & $1.000$ & $1.000$ \\
    additive, one-layer KAN        & $0.987$ & $0.996$ & $1.000$ & $1.000$ & $1.000$ & $1.000$ \\
    compositional, two-layer KAN   & $0.999$ & $1.000$ & $1.000$ & $1.000$ & $1.000$ & $1.000$ \\
    \midrule
    \multicolumn{7}{l}{\emph{largest relative coefficient error over $(a,b,c,q)$, oracle-library least squares}} \\
    Savitzky--Golay (oracle-tuned) & $0.000$ & $0.001$ & $0.011$ & $0.131$ & $0.412$ & $0.819$ \\
    linear (ridge on window)       & $0.000$ & $0.000$ & $0.000$ & $0.002$ & $0.034$ & $0.509$ \\
    additive, exact LS             & $0.000$ & $0.000$ & $0.000$ & $0.000$ & $0.000$ & $0.000$ \\
    additive, one-layer KAN        & $0.064$ & $0.016$ & $0.005$ & $0.001$ & $0.002$ & $0.009$ \\
    compositional, two-layer KAN   & $0.002$ & $0.001$ & $0.000$ & $0.000$ & $0.000$ & $0.000$ \\
    \midrule
    \multicolumn{7}{l}{\emph{STLSQ on the generic degree-two library: terms missed / spurious}} \\
    Savitzky--Golay (oracle-tuned) & $0/0$ & $0/0$ & $0/0$ & $0/0$ & $0/0$ & $3/0$ \\
    all learned front ends         & $0/0$ & $0/0$ & $0/0$ & $0/0$ & $0/0$ & $0/0$ \\
    \bottomrule
  \end{tabular}
  \caption{Genesio--Tesi discovery from a coarsely sampled scalar: the front end is what fails. Seed means over three seeds for the KAN maps; the exact-additive and linear maps are deterministic. At $\Delta=0.9$ the oracle-tuned Savitzky--Golay tower misses three of the four terms under STLSQ and its least-squares coefficients are wrong by up to $82\%$; every learned lift recovers all four terms, the exact-additive and compositional lifts to within $0.1\%$. %The one-layer KAN's residual error at dense sampling ($6\%$ on $a$ at $\Delta=0.05$) is an optimizer floor, not a class limit: the same additive class solved exactly by least squares is error-free at every $\Delta$.
  }
  \label{tab:genesio-undersampled}
\end{table}

\begin{figure}[h]
  \centering
  \includegraphics[width=0.95\linewidth]{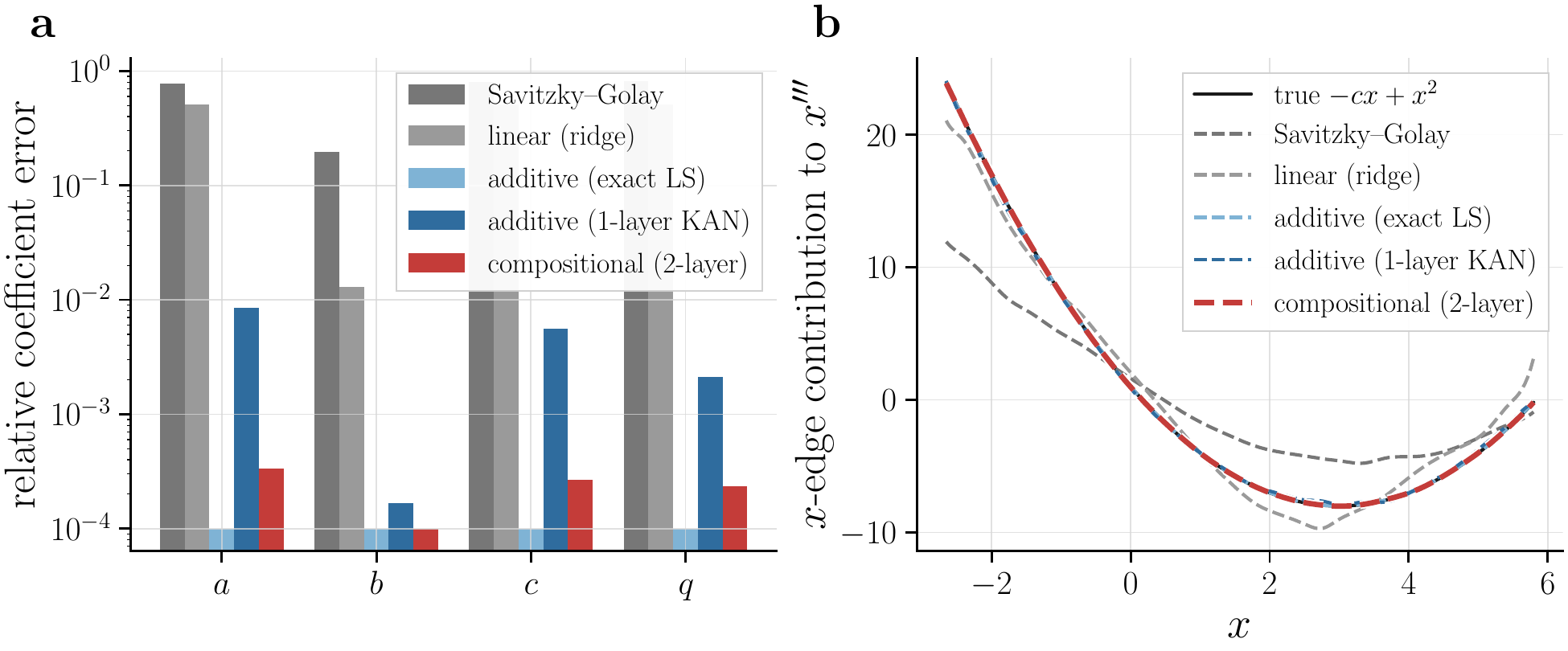}
  \caption{Genesio--Tesi at $\Delta=0.9$. \textbf{a:} relative coefficient error (oracle-library least squares) per front end; bars floored at $10^{-4}$. \textbf{b:} the library-free $x$-edge of the KANDy equation layer read off each front end's derivatives against the true $-cx+qx^{2}$; the learned lifts reproduce the quadratic, the fixed Savitzky--Golay formula does not.}
  \label{fig:genesio-undersampled}
\end{figure}

Table~\ref{tab:genesio-undersampled} and Fig.~\ref{fig:genesio-undersampled} give three results. First, the front end dominates: at $\Delta=0.9$ every learned lift, including the exactly solved additive class, recovers all four coefficients to within $1\%$ from the same data on which the oracle-tuned fixed formula is wrong by $82\%$ and sparse regression loses three terms. Transferring the delay-to-derivative map from a densely sampled calibration trajectory to a coarsely sampled deployment trajectory is exactly what a learned lift can do and a fixed formula cannot. Second, with derivatives held fixed, the sparse-regression backends agree: on a polynomial, correctly specified law, the oracle least squares and generic STLSQ recover the same equation from the same lifted derivatives. 
The library-free KANDy layer recovers the same four terms but with a seed spread of up to $14\%$ on the stiffness $a$ and $11\%$ on $c$ at $\Delta=0.9$ (Fig.~\ref{fig:genesio-undersampled}b shows the first seed), from collinearity among the three derivative edges---a back-end property independent of the front end, and the reason the coefficient rows of Table~\ref{tab:genesio-undersampled} use the oracle least squares. 
Third, depth is not needed here: the exact-additive lift ties the compositional one at every $\Delta$. This is the prediction of Lemma~\ref{lem:mixed} for a weakly nonlinear vector field---the only nonlinearity is $qx^{2}$, so $g''=2q$ is small, and the window-to-derivative map stays additively separable to well past the sampling interval at which fixed formulas fail (we find the exact-additive $R^{2}$ on $x'''$ first drops below $0.999$ at $\Delta\approx2$, near the aliasing wall). 
The compositional lift earns its keep on the strongly nonlinear Lorenz observable of Table~\ref{tab:undersampled-jet}, where the same exact-additive solve reaches only $R^{2}=0.62$ on $x'$ at $\Delta t=0.40$ against $0.88$ for the two-layer map. On $x''$, however, the exact-additive class is on par with the compositional one there ($0.968$ versus $0.970$ at $\Delta t=0.40$), so the depth requirement of Section~\ref{sec:delay} is a statement about the velocity map, which is what Lemma~\ref{lem:mixed} addresses; Table~\ref{tab:undersampled-jet}'s one-layer column is an Adam-trained KAN and sits $0.02$--$0.06$ below the exactly solved class.

%\clearpage
\section{Baseline configuration for the sine-Gordon benchmark}
\label{sec:supp-baselines}

On the PDE benchmark of Section~\ref{sec:res-wsindy} we compare pointwise SINDy/PDE-FIND~\citep{sindy,rudy2017data} and weak-form WSINDy~\citep{messenger2021weaksindy,messenger2021weaksindypde} on identical noisy data, both through their PySINDy implementations. SINDy/PDE-FIND estimates derivatives with a Savitzky--Golay filter and solves $u_{tt}=\Theta\boldsymbol{\xi}$ by sequential thresholded least squares (threshold $0.05$ on normalized columns) over polynomial powers of $u$ up to degree five together with $u_{xx}$; WSINDy avoids pointwise differentiation by integrating the equation against compactly supported test functions $\Psi(x,t)$ ($K=2000$ test-function centres) and transferring derivatives onto $\Psi$. WSINDy is run with both a naive polynomial library and an oracle library containing $\sin(u)$. We also compare D-CIPHER~\citep{kacprzyk2023dcipher}, which searches for $\sum_j \eta_j\,\mathcal{O}_j u = g(\mathbf{x},u)$ with scalar weights on the differential operators and only $g$ discovered in closed form; since $g$ may not depend on $u_t$, a saturating damping $\kappa\tanh(\nu u_t)$ lies outside its model class at any noise level, a structural ceiling measured in Note~\ref{sec:supp-tanh}. All methods are scored on the recovered restoring force $R(u)$ against $-\sin(u)$ over the observed range, which separates derivative-noise error from library misspecification.

%\clearpage
\section{Equation backends on identical derivatives: what the free-form edges adds}
\label{sec:supp-backends}

The sine-Gordon comparison of Section~\ref{sec:res-wsindy} pairs each method with its own derivative estimator. To isolate the equation representation, this note feeds every back end the same spline-jet samples $(u,u_{xx},u_t,u_{tt})$ that the Higher-Order-KANDy equation layer consumes---the surrogate of Eq.~\eqref{eq:fieldjet} fitted once per noise realization---and changes only the read-out. Five back ends: the free-form additive KAN layer (the paper's method, $\lambda=0$); PySINDy's STLSQ optimizer with the settings of the paper's baselines (threshold $0.05$ on normalized columns) on $[1,u,u^{2},u^{3},\sin u,u_{xx},u_t]$, i.e.\ sparse regression handed the $\sin(u)$ column; the same optimizer on the polynomial library of Table~\ref{tab:sg-term-support} without it; a SINDy-KAN~\citep{howard2026sindykan}; and a Symbolic--KAN~\citep{faroughi2026symbolickan}. SINDy-KAN has no released code; the Symbolic--KAN authors released code after their preprint, which we also ran (below). The SINDy-KAN row re-implements the mechanism of their Section~3: a B-spline KAN is trained while each edge is simultaneously regressed by STLSQ onto the univariate library $[x,x^{2},x^{3},\sin x,\cos x]$ of its own input (their Eq.~13), with a trainable shadow coefficient matrix carrying an $\ell_{1}$ penalty and the composite loss of their Eq.~16; the equation is the sparse per-edge representation read off the trained edges. The Symbolic--KAN row re-implements their Eqs.~5--13 and 22--29: units that apply one analytic primitive from $\{0,1,x,x^{2},x^{3},\sin,\cos,\tanh,e^{x},\log(1+|x|)\}$, chosen by an annealed Gumbel-softmax gate with entropy regularization, to a learned scalar projection of the inputs, summed with sigmoid unit gates under a budgeted sparsity penalty and hardened to one-hot at the end. Support is scored as in Table~\ref{tab:sg-term-support}: each method's own selection step (the STLSQ threshold for the library methods, activation pruning for the free-form layer, the per-edge sparse regression for the SINDy-KAN) followed by the $0.01$ coefficient tolerance; the Symbolic--KAN's hardened network is not additive in the inputs, so it is scored on the restoring curve $u\mapsto f(u,\bar{u}_{xx},\bar{u}_t)$ and the finite-difference slope in $u_{xx}$ only. Noise $\{0,1,5,10\}\%$, five independent seeds per level, $12{,}000$ interior collocation points, an independent noise field per (seed, level), as in the audit; the free-form layer here carries the three edges $(u,u_{xx},u_t)$ of Section~\ref{sec:res-wsindy}'s fit, without the two extra null controls of the audit.

\begin{table}[h]
  \centering\small
  \begin{tabular}{llcccc}
    \toprule
    noise & back end (all on the same spline jet) & form RMS & $\hat{c}$ ($u_{xx}$) & exact support \\
    \midrule
    $0\%$  & Higher-Order-KANDy (free-form edges) & $0.002\pm0.000$ & $1.000$ & $5/5$ \\
           & STLSQ, $\sin(u)$ column supplied     & $0.000\pm0.000$ & $1.000$ & $5/5$ \\
           & STLSQ, polynomial library            & $0.707\pm0.000$ & $0.376$ & $0/5$ \\
           & SINDy-KAN (re-impl.)                 & $0.029\pm0.046$ & $0.91\pm0.26$ & $0/5$ \\
           & Symbolic--KAN (re-impl.)             & $0.723\pm0.009$ & $0.208$ & --- \\
    \midrule
    $1\%$  & Higher-Order-KANDy (free-form edges) & $0.030\pm0.013$ & $0.993$ & $5/5$ \\
           & STLSQ, $\sin(u)$ column supplied     & $0.012\pm0.006$ & $0.990$ & $5/5$ \\
           & STLSQ, polynomial library            & $0.707\pm0.000$ & $0.377$ & $0/5$ \\
           & SINDy-KAN (re-impl.)                 & $0.049\pm0.040$ & $0.32\pm0.24$ & $0/5$ \\
           & Symbolic--KAN (re-impl.)             & $0.725\pm0.007$ & $0.201$ & --- \\
    \midrule
    $5\%$  & Higher-Order-KANDy (free-form edges) & $0.095\pm0.039$ & $0.964$ & $5/5$ \\
           & STLSQ, $\sin(u)$ column supplied     & $0.089\pm0.041$ & $0.917$ & $2/5$ \\
           & STLSQ, polynomial library            & $0.706\pm0.001$ & $0.377$ & $0/5$ \\
           & SINDy-KAN (re-impl.)                 & $0.105\pm0.042$ & $-0.44\pm0.34$ & $0/5$ \\
           & Symbolic--KAN (re-impl.)             & $0.726\pm0.013$ & $0.203$ & --- \\
    \midrule
    $10\%$ & Higher-Order-KANDy (free-form edges) & $0.230\pm0.037$ & $0.796$ & $5/5$ \\
           & STLSQ, $\sin(u)$ column supplied     & $0.170\pm0.041$ & $0.845$ & $0/5$ \\
           & STLSQ, polynomial library            & $0.706\pm0.002$ & $0.378$ & $0/5$ \\
           & SINDy-KAN (re-impl.)                 & $0.243\pm0.021$ & $-0.73\pm0.31$ & $0/5$ \\
           & Symbolic--KAN (re-impl.)             & $0.728\pm0.015$ & $0.165$ & --- \\
    \bottomrule
  \end{tabular}
  \caption{Equation back ends on identical spline-jet derivatives (sine-Gordon collision; mean $\pm$ std over five seeds; form RMS is the error of the recovered restoring curve against $-\sin(u)$ over the observed range; true $\hat{c}=1$; support: each method's own selection step, then the $0.01$ tolerance). Handed the $\sin(u)$ column, sparse regression matches or beats the free-form edge on form at every level, while its own selection retains a spurious linear $u$ ($-0.013$ to $-0.034$) on three seeds at $5\%$ and on every seed at $10\%$, as the weak-form oracle does in Table~\ref{tab:sine-gordon-equations}; without the column it fails by coverage at every level, including zero noise. The SINDy-KAN's per-edge symbolic library recovers $\sin(u)$ on the $u$ edge (shape correlation ${\geq}0.997$ at every level) but its read-out of the $u_{xx}$ edge is ill-conditioned---on the narrow $u_{xx}$ range the columns $x$, $x^{3}$, $\sin x$, $\cos x$ are nearly collinear---so $\hat{c}$ changes sign across seeds and the support is never exact. The Symbolic--KAN fits the data ($R^{2}$ $0.92$--$0.99$) but, with dense projections mixing all three inputs, isolates no term.}
  \label{tab:sg-backends}
\end{table}

\begin{figure}[h]
  \centering
  \includegraphics[width=0.95\linewidth]{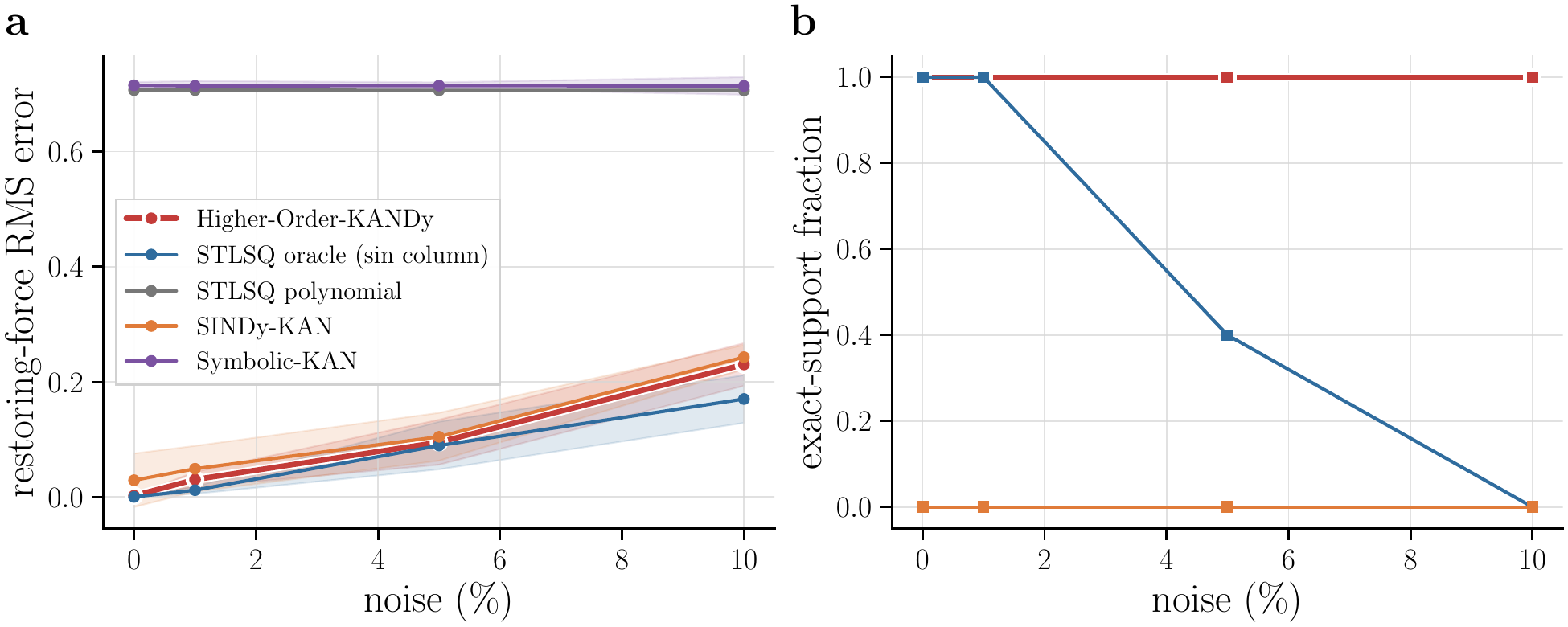}
  \caption{Back ends on identical derivatives. \textbf{a:} restoring-force RMS error versus noise (bands: $\pm1$ std over seeds). \textbf{b:} fraction of seeds with exact support (Symbolic--KAN omitted: non-additive).}
  \label{fig:sg-backends}
\end{figure}

Table~\ref{tab:sg-backends} and Fig.~\ref{fig:sg-backends} support two readings. First, the free-form edges add nothing to coefficient accuracy once sparse regression is handed the right column: the STLSQ oracle on the same derivatives is the better form estimator at every noisy level, as the WSINDy oracle is in Section~\ref{sec:res-wsindy}. Its advantage is that it does not require the column---the full-period $-\sin(u)$ is recovered from a spline edge with nothing named in advance, whereas the polynomial library fails by coverage at zero noise. Second, the two KAN-based discovery architectures fail for reasons that follow from their design rather than their tuning. SINDy-KAN discovers the equation by snapping each edge to a univariate library; on a coordinate whose observed range is narrow, the library is collinear, and the snap is ill-conditioned, the same range argument that limits polynomial libraries on the $u$ axis. Symbolic--KAN is a general symbolic regressor whose units act on learned projections of all inputs; nothing in it enforces that a term be a function of one derivative coordinate, so the additive-by-order structure that Higher-Order-KANDy imposes must instead be discovered, and on this problem is not. The comparison therefore does not show that one KAN parameterization is better than another; it shows that imposing one univariate edge per derivative coordinate is the right prior for the class of equations this paper targets. Both re-implementations are ours and are labeled as such.
We also ran the released Symbolic--KAN code\footnote{\texttt{github.com/sfaroughi3/Pub\_Symbolic\_KANs}, \texttt{Exp\_reaction\_diffusion}} with only its dataset generator replaced by the same jet samples, at the function defaults (two blocks of six units, $2{,}000$ Adam epochs followed by L-BFGS, default gate and unit-sparsity weights); its unit gates killed every unit at $1\%$ and $5\%$ noise (a constant output) and left one linear unit at $0\%$ and $10\%$ ($R^{2}$ $0.25$ and $0.40$). 
%The authors' own experiments use four blocks and $10{,}800$ epochs, so we regard this as an untuned run rather than a result, report it here for completeness, and score the re-implementation in the table.

% ------------------------------------------------------------------------

%\clearpage
\section{Discovered equations, sine-Gordon benchmark}
\label{sec:supp-equations}

\begin{figure}[t]
  \centering
  \includegraphics[width=\linewidth]{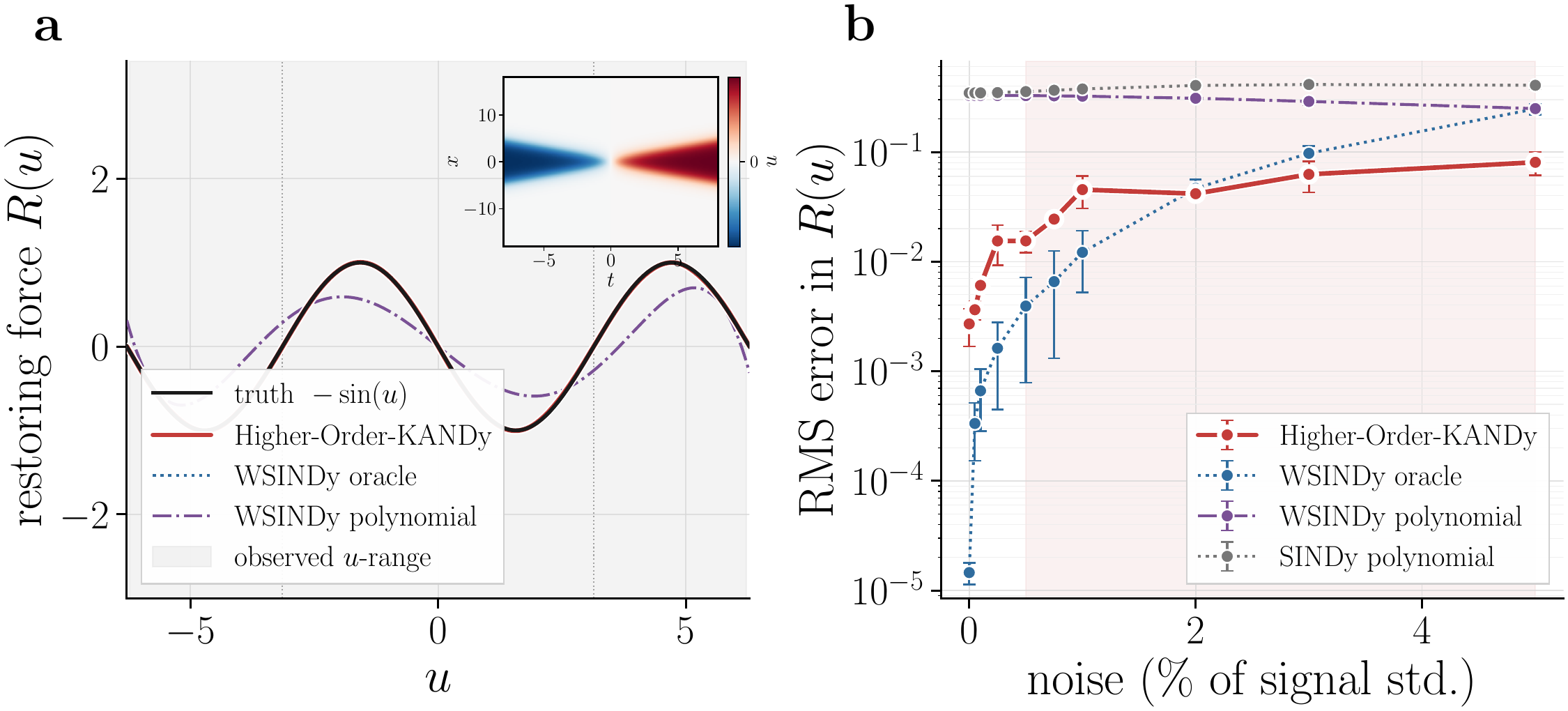}
  \caption{\textbf{Sine-Gordon kink--antikink collision: form recovery under
    noise.} \textbf{a:} the recovered restoring force $R(u)$ against
    $-\sin(u)$ over the observed $u$-range (shaded; inset: the collision
    field). The free-form Higher-Order-KANDy edge tracks the full-period
    $-\sin(u)$, while the polynomial-library recovery flattens outside a
    single period---the range limitation of Section~\ref{sec:res-wsindy},
    which no noise level fixes. \textbf{b:} RMS error of $R(u)$ versus noise:
    the WSINDy oracle, handed the exact $\sin(u)$ column, is the better form
    estimator at every noise level, the gap narrowing toward $5\%$.}
  \label{fig:sg-collision}
\end{figure}

Table~\ref{tab:sine-gordon-equations} provides the complete equations generated by each method applied to the kink-antikink collision equation, $u_{tt}=u_{xx}-\sin(u)$, at each of the twelve noise levels in the sweep found in Section~\ref{sec:res-wsindy}. The equation shown is the representative repeat (the first seed). The support column counts how many of the five repeats selected that exact term set. Every selected term is shown, including spurious terms; a term counts for every method only if its coefficient in the $u_{tt}$- normalized equation exceeds $0.01$---one percent of the unit-scale true coefficients---and smaller selected terms are omitted as numerically inactive. The correct column indicates the number of the two true terms ${u_{xx},\sin u}$ that were recovered. All rows come from the symmetric audit protocol: the Higher-Order-KANDy equation layer carries five free-form edges---$u$, $u_{xx}$, and the null controls $u_t$, $u_x$, $u_{xt}$---with no parsimony penalty on any edge ($\lambda=0$), an independent noise realization is drawn per (repeat, level), and the network initialization varies with the repeat seed. A $\lambda$-sensitivity audit confirms the support selection owes nothing to a sparsity prior---indeed the opposite: adding a uniform L1 penalty on all five edges destroys recovery rather than aiding it (at $\lambda=0.1$ the discovered equation collapses to $u_{tt}=0.23\,u_{xx}$ at every noise level tested, and at $\lambda=0.3$ to $u_{tt}=0.10\,u_{xx}$), so the active-set selection is a property of the learned free-form edges and the audit gates, not of a parsimony penalty.
D-CIPHER is a reduced diagnostic (one seed, four noise levels, original protocol) re-scored at the same coefficient tolerance.

{\scriptsize
\setlength{\tabcolsep}{3pt}
\begin{longtable}{@{}l>{\raggedright\arraybackslash}p{3.0cm}>{\raggedright\arraybackslash}p{5.9cm}c>{\raggedright\arraybackslash}p{1.6cm}c@{}}
\caption{Equations discovered for the sine-Gordon kink--antikink collision, $u_{tt}=u_{xx}-\sin(u)$, at every noise level\rev{ of the symmetric audit protocol (noise to $10\%$; support = repeats out of five selecting exactly the displayed term set, which is the first seed's and coincides with the true set only where the spurious column is empty)}. Higher-Order-KANDy rows are set in bold; the grey D-CIPHER rows are the reduced one-seed diagnostic and carry no support
count.}
\label{tab:sine-gordon-equations}\\
\toprule
noise & method & discovered equation & correct & spurious & support \\
\midrule
\endfirsthead
\multicolumn{5}{@{}l}{\emph{Table~\ref{tab:sine-gordon-equations}, continued}}\\
\toprule
noise & method & discovered equation & correct & spurious & support \\
\midrule
\endhead
\midrule
\multicolumn{5}{r@{}}{\emph{continued on next page}}\\
\endfoot
\bottomrule
\endlastfoot
0\% & \textbf{Higher-Order-KANDy} & $u_{tt} = -1.000\,\sin(u) + 1.001\,u_{xx}$ & 2/2 & -- & 5/5 \\
 & WSINDy oracle & $u_{tt} = -1.000\,\sin(u) + 1.000\,u_{xx}$ & 2/2 & -- & 5/5 \\
 & \textcolor{gray}{D-CIPHER (GP)} & $u_{tt} = -0.891\,\sin(u) + 0.794\,u_{xx}$ & 2/2 & -- & -- \\
 & WSINDy polynomial & $u_{tt} = -0.520\,u + 0.053\,u^{3} - 0.024\,u_t + 0.775\,u_{xx}$ & 1/2 & $u$, $u^{3}$, $u_t$ & 3/5 \\
 & SINDy polynomial & $u_{tt} = -0.437\,u + 0.044\,u^{3} + 0.692\,u_{xx}$ & 1/2 & $u$, $u^{3}$ & 5/5 \\
\midrule
0.05\% & \textbf{Higher-Order-KANDy} & $u_{tt} = -1.000\,\sin(u) + 1.003\,u_{xx}$ & 2/2 & -- & 5/5 \\
 & WSINDy oracle & $u_{tt} = -1.000\,\sin(u) + 1.000\,u_{xx}$ & 2/2 & -- & 5/5 \\
 & WSINDy polynomial & $u_{tt} = -0.570\,u + 0.058\,u^{3} + 0.015\,u_t + 0.805\,u_{xx}$ & 1/2 & $u$, $u^{3}$, $u_t$ & 3/5 \\
 & SINDy polynomial & $u_{tt} = -0.426\,u + 0.043\,u^{3} + 0.677\,u_{xx}$ & 1/2 & $u$, $u^{3}$ & 5/5 \\
\midrule
0.10\% & \textbf{Higher-Order-KANDy} & $u_{tt} = -1.002\,\sin(u) + 1.005\,u_{xx}$ & 2/2 & -- & 5/5 \\
 & WSINDy oracle & $u_{tt} = -1.000\,\sin(u) + 1.000\,u_{xx}$ & 2/2 & -- & 5/5 \\
 & WSINDy polynomial & $u_{tt} = -0.536\,u + 0.055\,u^{3} - 0.016\,u_t + 0.783\,u_{xx}$ & 1/2 & $u$, $u^{3}$, $u_t$ & 3/5 \\
 & SINDy polynomial & $u_{tt} = -0.396\,u + 0.039\,u^{3} + 0.635\,u_{xx}$ & 1/2 & $u$, $u^{3}$ & 5/5 \\
\midrule
0.25\% & \textbf{Higher-Order-KANDy} & $u_{tt} = -0.981\,\sin(u) + 0.988\,u_{xx}$ & 2/2 & -- & 5/5 \\
 & WSINDy oracle & $u_{tt} = -1.002\,\sin(u) + 1.000\,u_{xx}$ & 2/2 & -- & 5/5 \\
 & WSINDy polynomial & $u_{tt} = -0.537\,u + 0.055\,u^{3} + 0.016\,u_t + 0.787\,u_{xx}$ & 1/2 & $u$, $u^{3}$, $u_t$ & 2/5 \\
 & SINDy polynomial & $u_{tt} = -0.261\,u + 0.023\,u^{3} + 0.445\,u_{xx}$ & 1/2 & $u$, $u^{3}$ & 5/5 \\
\midrule
0.50\% & \textbf{Higher-Order-KANDy} & $u_{tt} = -0.966\,\sin(u) + 0.991\,u_{xx}$ & 2/2 & -- & 5/5 \\
 & WSINDy oracle & $u_{tt} = -1.000\,\sin(u) + 1.000\,u_{xx}$ & 2/2 & -- & 5/5 \\
 & \textcolor{gray}{D-CIPHER (GP)} & $u_{tt} = -0.843\,\sin(u) + 0.764\,u_{xx}$ & 2/2 & -- & -- \\
 & WSINDy polynomial & $u_{tt} = -0.487\,u + 0.049\,u^{3} + 0.735\,u_{xx}$ & 1/2 & $u$, $u^{3}$ & 1/5 \\
 & SINDy polynomial & $u_{tt} = -0.102\,u + 0.222\,u_{xx}$ & 1/2 & $u$ & 5/5 \\
\midrule
0.75\% & \textbf{Higher-Order-KANDy} & $u_{tt} = -0.982\,\sin(u) + 0.998\,u_{xx}$ & 2/2 & -- & 5/5 \\
 & WSINDy oracle & $u_{tt} = -1.001\,\sin(u) + 1.000\,u_{xx}$ & 2/2 & -- & 5/5 \\
 & WSINDy polynomial & $u_{tt} = -0.562\,u + 0.057\,u^{3} - 0.014\,u_t + 0.818\,u_{xx}$ & 1/2 & $u$, $u^{3}$, $u_t$ & 2/5 \\
 & SINDy polynomial & $u_{tt} = -0.032\,u + 0.127\,u_{xx}$ & 1/2 & $u$ & 5/5 \\
\midrule
1\% & \textbf{Higher-Order-KANDy} & $u_{tt} = -0.969\,\sin(u) + 0.984\,u_{xx}$ & 2/2 & -- & 5/5 \\
 & WSINDy oracle & $u_{tt} = -0.994\,\sin(u) + 1.000\,u_{xx}$ & 2/2 & -- & 5/5 \\
 & \textcolor{gray}{D-CIPHER (GP)} & $u_{tt} = -0.857\,\sin(u) - 0.011\,u_{xt} + 0.770\,u_{xx}$ & 2/2 & $u_{xt}$ & -- \\
 & WSINDy polynomial & $u_{tt} = -0.500\,u + 0.050\,u^{3} + 0.749\,u_{xx}$ & 1/2 & $u$, $u^{3}$ & 4/5 \\
 & SINDy polynomial & $u_{tt} = 0.083\,u_{xx}$ & 1/2 & -- & 4/5 \\
\midrule
2\% & \textbf{Higher-Order-KANDy} & $u_{tt} = -0.967\,\sin(u) + 1.003\,u_{xx}$ & 2/2 & -- & 5/5 \\
 & WSINDy oracle & $u_{tt} = -1.002\,\sin(u) + 0.999\,u_{xx}$ & 2/2 & -- & 5/5 \\
 & WSINDy polynomial & $u_{tt} = -0.520\,u + 0.053\,u^{3} - 0.019\,u_t + 0.768\,u_{xx}$ & 1/2 & $u$, $u^{3}$, $u_t$ & 4/5 \\
 & SINDy polynomial & $u_{tt} = 0.036\,u_{xx}$ & 1/2 & -- & 4/5 \\
\midrule
3\% & \textbf{Higher-Order-KANDy} & $u_{tt} = -0.984\,\sin(u) + 1.023\,u_{xx}$ & 2/2 & -- & 5/5 \\
 & WSINDy oracle & $u_{tt} = -1.009\,\sin(u) + 1.001\,u_{xx}$ & 2/2 & -- & 4/5 \\
 & WSINDy polynomial & $u_{tt} = -0.531\,u + 0.054\,u^{3} + 0.792\,u_{xx}$ & 1/2 & $u$, $u^{3}$ & 2/5 \\
 & SINDy polynomial & $u_{tt} = -0.047\,u + 0.025\,u_{xx}$ & 1/2 & $u$ & 5/5 \\
\midrule
5\% & \textbf{Higher-Order-KANDy} & $u_{tt} = -0.790\,\sin(u) + 0.879\,u_{xx}$ & 2/2 & -- & 5/5 \\
 & WSINDy oracle & $u_{tt} = -0.012\,u - 0.998\,\sin(u) + 1.007\,u_{xx}$ & 2/2 & $u$ & 3/5 \\
 & \textcolor{gray}{D-CIPHER (GP)} & $u_{tt} = -1.085\,\sin(u) + 0.021\,u_{xt} + 0.842\,u_{xx}$ & 2/2 & $u_{xt}$ & -- \\
 & WSINDy polynomial & $u_{tt} = -0.548\,u + 0.056\,u^{3} + 0.793\,u_{xx}$ & 1/2 & $u$, $u^{3}$ & 2/5 \\
 & SINDy polynomial & $u_{tt} = -0.172\,u + 0.019\,u_{xx}$ & 1/2 & $u$ & 5/5 \\
\midrule
7.5\% & \textbf{Higher-Order-KANDy} & $u_{tt} = -0.670\,\sin(u) + 0.774\,u_{xx}$ & 2/2 & -- & 5/5 \\
 & WSINDy oracle & $u_{tt} = -1.018\,\sin(u) + 1.013\,u_{xx}$ & 2/2 & -- & 4/5 \\
 & WSINDy polynomial & $u_{tt} = -0.546\,u + 0.055\,u^{3} - 0.020\,u_t + 0.788\,u_{xx}$ & 1/2 & $u$, $u^{3}$, $u_t$ & 2/5 \\
 & SINDy polynomial & $u_{tt} = -0.413\,u + 0.028\,u^{3} + 0.018\,u_{xx}$ & 1/2 & $u$, $u^{3}$ & 3/5 \\
\midrule
10\% & \textbf{Higher-Order-KANDy} & $u_{tt} = -0.695\,\sin(u) + 0.801\,u_{xx}$ & 2/2 & -- & 4/5 \\
 & WSINDy oracle & $u_{tt} = -0.015\,u - 0.988\,\sin(u) + 0.989\,u_{xx}$ & 2/2 & $u$ & 3/5 \\
 & WSINDy polynomial & $u_{tt} = -0.545\,u + 0.055\,u^{3} + 0.787\,u_{xx}$ & 1/2 & $u$, $u^{3}$ & 1/5 \\
 & SINDy polynomial & $u_{tt} = -0.694\,u + 0.054\,u^{3} + 0.016\,u_{xx}$ & 1/2 & $u$, $u^{3}$ & 3/5 \\
\end{longtable}
}

\end{document}